\documentclass{jonasart}

\usepackage{amscd}

\theoremstyle{plain}
\newtheorem*{notation}{Notation}

\DeclareMathOperator{\Ann}{Ann}
\DeclareMathOperator{\Aut}{Aut}
\DeclareMathOperator{\B}{B}
\DeclareMathOperator{\C}{C}
\DeclareMathOperator{\Ext}{Ext}
\DeclareMathOperator{\Fix}{Fix}
\DeclareMathOperator{\Ho}{H}
\DeclareMathOperator{\Id}{id}
\DeclareMathOperator{\im}{im}
\DeclareMathOperator{\Ker}{ker}
\DeclareMathOperator{\Z}{Z}

\allowdisplaybreaks

\title{Skew braces and Rota--Baxter operators on semi-direct products}

\authors{Pragya Belwal and Mahender Singh}

\authorinfo[P. Belwal]{Indian Institute of Science Education and Research Mohali, India}{pragyabelwal.math@gmail.com}

\authorinfo[M. Singh]{Indian Institute of Science Education and Research Mohali, India}{mahender@iisermohali.ac.in}

\abstract{This paper examines the connections between (relative) Rota--Baxter groups, skew left braces, and enlargements of these structures on naturally associated semi-direct products. Given a skew left brace, we define a new skew left brace, referred to as its square, on the natural semi-direct product of its additive and multiplicative groups. Further, the square construction is distinct from the previously known double construction arising as a special case of matched pairs of skew braces. This provides a method to construct a new bijective, non-degenerate solution to the Yang--Baxter equation from an existing solution arising from a skew left brace. We show that the square construction is functorial and integrates naturally into both the cohomological and extension-theoretic frameworks for (relative) Rota--Baxter groups and skew left braces. Furthermore, we provide a sufficient condition under which two isoclinic skew left braces yield isoclinic squares.}



\editinfo{January 3, 2026}{March 24, 2026}{Leandro Vendramin}

\acknowledgments{The authors thank the anonymous referees for their valuable suggestions, which have improved the clarity of the paper. PB is supported by the UGC Senior Research Fellowship at IISER Mohali.  MS is supported by the Anusandhan National Research Foundation grant ANRF/\-ARGM/\-2025/\-000406/\-MTR.}

\VOLUME{1}
\ISSUE{1}
\NUMBER{1}
\DOI{https://doi.org/10.46298/jonas.17223}

\begin{document}
	
	\section{Introduction}
	
	The program for classifying set-theoretic solutions to the Yang--Baxter equation was proposed by Drinfel'd in the 1990s~\cite{MR1183474}. A key result by Guarnieri and Vendramin~\cite{MR3647970} established that every skew left brace yields a~non-degenerate bijective set-theoretic solution to the Yang--Baxter equation. Recently, Guo, Lang, and Sheng~\cite{MR4271483} introduced Rota--Baxter operators on Lie groups as a~means of producing such operators on the corresponding Lie algebras. The study of Rota--Baxter operators on more general groups, beyond the Lie group setting, was undertaken by Bardakov and Gubarev~\cite{MR4370524, MR4556953}, who proved that every Rota--Baxter operator on a~group induces a~skew left brace structure on that group. This result links Rota--Baxter operators on groups directly to Drinfel'd's program. At the same time, Caranti and Stefanello~\cite{MR4531710} gave a~cohomological characterization of when a~skew left brace can be induced from a~Rota--Baxter group. The idea of Rota--Baxter operators was further generalized by Jiang, Sheng, and Zhu~\cite{MR4704155} to relative Rota--Baxter operators on Lie groups. Building on this, it was proved in~\cite{MR4857550} that every skew left brace is induced by a~relative Rota--Baxter group, and that there is an isomorphism between the categories of bijective relative Rota--Baxter groups and skew left braces. In~\cite{MR4705761}, Bai et al. further connected relative Rota--Baxter groups to Butcher groups, post-groups, and pre-groups. As another generalization, Catino et al.~\cite{MR4412793, MR4551911} introduced Rota--Baxter operators on Clifford semigroups and established a~correspondence between such operators and weak braces, the latter being known to produce set-theoretic solutions to the Yang--Baxter equation that may be degenerate.
	
	A cohomology theory for relative Rota--Baxter operators on Lie groups was introduced in~\cite{MR4704155}. However, this theory is not applicable to abstract groups, as it is defined in terms of the cohomology of the descendent group with coefficients in the Lie algebra of the acting Lie group, viewed as a~module over the descendent group. To address this limitation, a~cohomology and extension theory for relative Rota--Baxter groups was developed in~\cite{MR4819014}, enabling the classification of such extensions via the second cohomology group. The relationship between the cohomology of relative Rota--Baxter groups and that of their associated skew left braces was also investigated. These constructions were subsequently applied in~\cite{MR4755474} to define the Schur multiplier and Schur covers of relative Rota--Baxter groups, and their connections to the corresponding constructions for skew left braces were explored.
	
	Given the central role of skew left braces in classifying set-theoretic solutions to the Yang--Baxter equation, developing new methods for their construction is of fundamental importance. In this paper, we define a~novel construction, called the square of a~skew left brace, which produces a~skew left brace from a~given one. This construction is defined on the natural semi-direct product of the additive and multiplicative groups of the given skew left brace. The idea is motivated by an observation by Bardakov et al.~\cite{MR4868162} that every relative Rota--Baxter group gives rise to a~Rota--Baxter operator on the natural semi-direct product of its constituent groups. We observe that our square construction is distinct from the previously known double construction of a~skew left brace due to Smoktunowicz and Vendramin~\cite{MR3763907}. This provides a~method to construct a~new bijective, non-degenerate solution to the Yang--Baxter equation from an existing solution arising from a~skew left brace. Further, we show that the square construction is functorial and naturally integrates into both the cohomological and extension-theoretic frameworks for (relative) Rota--Baxter groups and skew left braces.
	
	The paper is organised as follows. In Section~\ref{section prelim}, we recall some necessary background from the extension and cohomology theory of skew left braces, Rota--Baxter groups and relative Rota--Baxter groups. In Section~\ref{section RRB to RRB on SDP}, we prove that the construction of the Rota--Baxter group on the associated semi-direct product arising from a~given relative Rota--Baxter group is functorial (Proposition~\ref{RRB to RB functor}), and yields a~homomorphism from the second cohomology of the relative Rota--Baxter group to that of the induced Rota--Baxter group (Proposition~\ref{cohom RRB to cohom RB}). In Section~\ref{section SLB to SLB on SDP}, we give the construction of the square of a~skew left brace, and illustrate how it differs from its double. We prove that the construction of the square of a~skew left brace is functorial (Proposition~\ref{sb iso to isb iso}), and that there is a~homomorphism from the second cohomology of the given skew left brace to that of its square (Proposition~\ref{homology homo SLB}). In Section~\ref{section comm diag homol groups}, we prove that, for each skew left brace, the homomorphisms of cohomology groups defined in the preceding sections give rise to a~commutative diagram of cohomology groups of (relative) Rota--Baxter groups, skew left braces, and their squares (Theorem~\ref{theorem comm diag homol groups}). Finally, in Section~\ref{section isoclinism of squares}, we provide a~sufficient condition under which two isoclinic skew left braces yield isoclinic squares (Theorem~\ref{theorem isoclinicsdp}).
	
	\section{Preliminaries}\label{section prelim}
	In this section, we recall some essential results on abelian extensions of skew left braces and (relative) Rota--Baxter groups.
	
	\subsection{Abelian extensions of skew left braces}
	We begin by recalling some basic ideas about abelian extensions of skew left braces.
	
	\begin{definition}
		A skew left brace is a~triple $(H,\cdot,\circ)$, where $(H,\cdot)$ and $(H, \circ)$ are groups such that
		\[
		a \circ (b \cdot c)=(a\circ b) \cdot a^{-1} \cdot (a \circ c)
		\]
		for all $a,b,c \in H$, where $a^{-1}$ denotes the inverse of $a$ in $(H, \cdot)$. The groups $(H,\cdot)$ and $(H, \circ)$ are called the additive and the multiplicative groups of the skew left brace $(H, \cdot, \circ)$, respectively. When the additive group is abelian, then $(H,\cdot,\circ)$ is simply called a~left brace.
	\end{definition}
	
	Any group can be viewed as a~skew left brace with the same additive and the multiplicative group, called the \textit{trivial skew left brace}. Given a~skew left brace $(H,\cdot,\circ)$, there is an associated action $\lambda^H:(H,\circ)\rightarrow \Aut(H,\cdot)$, which is a~group homomorphism defined by
	\[
	\lambda^H_a(b)=a^{-1}\cdot(a\circ b)
	\]
	for all $a, b \in H$. The associated action plays a~key role in understanding skew left braces.
	\begin{notation}
		For simplicity in notation, we make the following conventions:
		\begin{enumerate}
			\item We denote the additive and the multiplicative groups of a~skew left brace $(H, \cdot, \circ)$ by $H ^{(\cdot)}$ and $H ^{(\circ)}$, respectively.
			\item When there is no ambiguity, we write the group operation in $H ^{(\cdot)}$ simply as $ab$.
			\item We denote the inverse of an element $g$ in $H ^{(\cdot)}$ by $g^{-1}$, and the inverse of $g$ in $H ^{(\circ)}$ by $g^\dagger$.
			\item When the context is clear, we will not differentiate between the notations for the additive and the multiplicative groups of different skew left braces.
			\item We will sometimes denote the image $\phi(x)$ of a~homomorphism $\phi$ by $\phi_x$.
			\item We denote by $\textbf{1}$ the trivial skew left brace for which both the underlying groups are trivial.
		\end{enumerate}
	\end{notation}
	
	Next, we introduce homomorphisms, ideals, and extensions of skew left braces.
	
	\begin{definition}
		Let $(H, \cdot_H, \circ_H )$ and $(K, \cdot_K, \circ_K )$ be skew left braces. A map $f: H \rightarrow K$ is called a~homomorphism of skew left braces if
		\[
		f(a \cdot_H b)=f(a) \cdot_K f(b) \quad \textrm{and} \quad f(a \circ_H b)=f(a) \circ_K f(b)
		\]
		for all $a, b \in H$.
	\end{definition}
	
	\begin{definition}
		Let $(H,\cdot,\circ)$ be a~skew left brace. A normal subgroup $I$ of both $(H,\cdot)$ and $(H,\circ)$ is said to be an ideal of $(H,\cdot,\circ)$ if $\lambda^H_h (I)\subseteq I$ for all $h\in H$.
	\end{definition}
	
	\begin{definition}
		Let $(I, \cdot)$ be an abelian group viewed as a~trivial brace and $(M, \cdot,\circ)$ be a~skew left brace. An (abelian) extension of $(M, \cdot,\circ)$ by $(I,\cdot)$ is a~skew left brace $(E,\cdot,\circ)$ that fits into the sequence
		\[
		\mathcal{E} : \quad \quad {\bf 1} \to (I,\cdot) \stackrel{i}{\to}  (E,\cdot,\circ) \stackrel{\pi}{\to} (M,\cdot,\circ) \to {\bf 1}
		\]
		where $i$ and $\pi$ are homomorphisms of skew left braces such that
		$i$ is injective, $\pi$ is surjective and $\im(i)=\ker(\pi)$.
	\end{definition}
	
	For simplicity, we denote $i(y)$ by $y$. By a~(normalised) set-theoretic section to $\mathcal{E}$, we mean a~map $s: H \rightarrow E$ such that $\pi \, s= \Id_{H}$ and $s(1)=1$. The equivalence of extensions of skew left braces is defined analogously to the equivalence of extensions of groups and other algebraic structures. 	We recall essential results on extensions of skew left braces by abelian groups from~\cite{MR4604853} (see also~\cite{MR3530867}).
	
	Let $(M, \cdot, \circ)$ be a~skew left brace and
	\[
	\mathcal{E}: \quad \quad {\bf 1} \longrightarrow (I, \cdot) \stackrel{i}{\longrightarrow}  (E,\cdot,\circ) \stackrel{\pi}{\longrightarrow} (M,\cdot,\circ) \longrightarrow {\bf 1}
	\]
	be an extension of $(M, \cdot, \circ)$ by the trivial brace $(I, \cdot)$. Let $s: M \rightarrow E$ be a~set-theoretic section to $\mathcal{E}$. We define maps $\xi, 	\epsilon: M^{(\circ)} \rightarrow \Aut(I)$ and $\zeta:M^{(\cdot)} \rightarrow \Aut(I)$ by
	\begin{eqnarray}
		\xi_m (y) & = & \lambda^E_{s(m)}(y),\label{action1 sb}\\
		\zeta_m (y) & = & s(m)^{-1} \cdot y \cdot s(m)~\label{action2 sb}~\textrm{and}\\
		\epsilon_m (y) & = & s(m)^{\dagger} \circ y \circ s(m)\label{action3 sb}
	\end{eqnarray}
	for $m \in M$ and $y \in I$, where $x^{-1}$ and $x^\dagger$ denote the inverse of $x$ in $E^{(\cdot)}$ and $E^{(\circ)}$, respectively. It is not difficult to see that the map $\xi$ is a~homomorphism, whereas the maps $\zeta$ and $\epsilon $ are anti-homomorphisms. Furthermore, these maps are independent of the choice of the set-theoretic section~\cite[Proposition 3.4]{MR4604853}. The triplet $(\xi, \zeta, \epsilon)$ is called the \emph{associated action} of the extension $\mathcal{E}$.
	
	Next, recall the definition of the second cohomology group of a~skew left brace $(M, \cdot, \circ)$ with coefficients in an abelian group $(I, \cdot)$ viewed as a~trivial brace. Let $\xi: M ^{(\circ)} \rightarrow \Aut (I)$ be a~homomorphism and $\zeta: M^{(\cdot)} \rightarrow \Aut(I)$ and $\epsilon: M^{(\circ)} \rightarrow \Aut(I)$ be anti-homo\-morphisms satisfying the conditions
	\begin{eqnarray*}
		\xi_{m_1 \cdot m_2}(\epsilon_{m_1 \cdot m_2}(y)) \,\zeta_{m_2}(y) & = & \zeta_{m_2}(\xi_{ m_1 }(\epsilon_{m_1}(y))) \, \xi_{ m_2} (\epsilon_{m_2}(y)) \textrm{and}\\
		\zeta_{m^{-1}_1 \cdot (m_1 \circ m_2 )}(\xi_{m_1}(y)) & =& \xi_{m_1}(\zeta_{m_2}(y))
	\end{eqnarray*}
	for all $m_1, m_2 \in M$ and $y \in I$. Such a~triplet $(\xi, \zeta, \epsilon)$ is referred to as a~\emph{good triplet} of action of $M$ on $I$.
	
	Let $g,f: M \times M \rightarrow I$ be maps satisfying
	\begin{gather}
		g(m_2, m_3 ) \, g(m_1 \cdot m_2, m_3 )^{-1} \, g(m_1, m_2 \cdot m_3 ) \, \zeta_{m_3}(g(m_1, m_2 ))^{-1} = 1, \label{sbcocycle1}
	\end{gather}
	\begin{align}
		\xi_{m_1}(f(m_2, m_3 )){ }&{ }f(m_1 \circ m_2, m_3 )^{-1} \nonumber \\ &{ }f(m_1, m_2 \circ m_3 ) \, \xi_{m_1 \circ m_2 \circ m_3} \, (\epsilon_{m_3}(\xi^{-1}_{m_1 \circ m_2}f(m_1, m_2 )))^{-1} = 1, \label{sbcocycle2}
	\end{align}
	\begin{align}
		\xi_{m_1}(g(m_2, m_3 )){ }&{ }\zeta_{m_1 \circ m_3} (g(m_1, m^{-1}_1 )) \nonumber \\
		&{ } \zeta_{m_1 \circ m_3} (g(m_1 \circ m_2, m^{-1}_1 ))^{-1} \, g((m_1 \circ m_2 ) m^{-1}_1, m_1 \circ m_3 )^{-1} \nonumber \\
		&{ } \zeta_{- m_1\cdot (m_1 \circ m_3 )} \, (f(m_1, m_2 ))^{-1} \, f(m_1, m_2 \cdot m_3 ) \, f(m_1, m_3 )^{-1} = 1, \label{sbcocycle3}
	\end{align}
	for all $m_1, m_2, m_3 \in M$. Let $\Z_{SB}^2 (M, I)$ be the set formed by pairs $(g,f)$ of functions $g,f: M \times M \rightarrow I$ that satisfy~\eqref{sbcocycle1},~\eqref{sbcocycle2},~\eqref{sbcocycle3} and vanish on degenerate tuples, and let $\B_{SB}^2 (M, I)$ be the set of pairs $(g, f) \in \Z_{SB}^2 (M, I)$ such that there exists a~map $\theta:M \to I$ satisfying
	\begin{eqnarray}
		g(m_1, m_2 ) &=&\theta(m_1\cdot m_2 )^{-1} \zeta_{m_2}(\theta(m_1 ))\theta(m_2 ) \label{boundary con sb1} \textrm{and}\\
		f(m_1, m_2 ) &=& \theta(m_1\circ m_2 )^{-1}\xi_{m_1\circ m_2} (\epsilon_{m_2}(\xi^{-1}_{m_1}(\theta(m_1 ))))\xi_{m_1}(\theta(m_2 )) \label{boundary con sb2}
	\end{eqnarray}
	for all $m_1, m_2 \in M$. Then the second cohomology group of $(M, \cdot, \circ)$ with coefficients in $I$ corresponding to the given good triplet of actions $(\xi, \zeta, \epsilon)$ is defined as
	\[
	\Ho^2_{SB}(M, I) =  \Z_{SB}^2(M, I)/\B_{SB}^2(M, I).
	\]
	
	Let $\Ext_{(\xi, \zeta, \epsilon)}(M, I)$ denote the set of equivalence classes of those skew left brace extensions of $M$ by $I$ whose corresponding triplet of actions is $(\xi, \zeta, \epsilon)$. Then the following result holds~\cite[Theorem A]{MR4604853}.
	
	\begin{theorem}\label{gbij-thm sb}
		Let $(M, \cdot, \circ)$ be a~skew left brace and $(I, \cdot)$ an abelian group viewed as a~trivial brace. Then there is a~bijection $\Upsilon:\Ext_{(\xi, \zeta, \epsilon)}(M, I) \rightarrow \Ho^2_{SB}(M, I)$ given by $\Upsilon([\mathcal{E}])=[\tau_1, \tau_2 ]$, where
		\begin{eqnarray}
			\tau_1 (m_1, m_2 ) &= & s(m_1 \cdot m_2 )^{-1} \cdot s(m_1 ) \cdot s(m_2 )\label{tau sb}~\textrm{and}\\
			\tau_2 (m_1, m_2 ) &=& s(m_1 \circ m_2 )^{-1} \cdot \big(s(m_1 ) \circ s(m_2 )\big)\label{tildetau sb}
		\end{eqnarray}
		for all $m_1, m_2 \in M$ and $s$ is a~set-theoretic section to $\mathcal{E}$.
	\end{theorem}
	
	\subsection{Abelian extensions of Rota--Baxter groups}\label{RB cohomology}
	Next, we recall some necessary results on Rota--Baxter groups from~\cite[Section 3, 4]{MR4644858}.
	
	\begin{definition}
		Let $(G, \cdot)$ be a~group. A map $R: G \rightarrow G$ is called a~Rota--Baxter operator of weight $1$ on $G$ if
		\[
		R(x) \cdot R(y)= R \big(x \cdot R(x) \cdot y \cdot R(x)^{-1} \big),
		\]
		for all $x, y \in G$. 	A group $(G, \cdot)$ equipped with a~Rota--Baxter operator of weight $1$ is called a~Rota--Baxter group, and is denoted by the pair $(G,R)$.
	\end{definition}
	
	The following result connects Rota--Baxter groups to skew left braces~\cite[Proposition~3.1]{MR4370524}.
	
	\begin{proposition}
		Let $(G,\cdot)$ be a~group and $R: G \to G$ be a~Rota--Baxter operator. If we define $x \circ_R y = x \cdot R(x) \cdot y \cdot R(x)^{-1}$, then $(G, \cdot, \circ_R )$ is a~skew left brace.
	\end{proposition}
	
	Note that, if the skew left brace $(G, \cdot, \circ_R )$ is induced by a~Rota--Baxter operator $R: G \to G$, then the associated action is given by
	\begin{equation}\label{RB induced SLB associated action}
		\lambda^G_x (y)= R(x) \cdot y \cdot R(x)^{-1},
	\end{equation}
	which is simply the conjugation action by the image of $R$.
	
	\begin{definition}
		Let $(G, R_{G})$ and $(H, R_{H})$ be Rota--Baxter groups. A map $\phi: G \rightarrow H$ is called a~homomorphism of Rota--Baxter groups if $\phi$ is a~group homomorphism satisfying
		\[
		\phi \, R_{G}=  R_{H} \, \phi.
		\]
	\end{definition}
	
	\begin{definition}
		Let $(I, R_I )$ and $(H, R_H )$ be Rota--Baxter groups. A Rota--Baxter extension of $(H, R_H )$ by $(I, R_I )$ is a~Rota--Baxter group $(E,R_E )$ that fits into the sequence
		\[
		\mathcal{E}: \quad \quad 1 \to (I,R_I) \stackrel{i}{\to}  (E,R_E) \stackrel{\pi}{\to} (H,R_H) \to 1,
		\]
		where $i$ and $\pi$ are homomorphisms of Rota--Baxter groups such that $i$ is injective, $\pi$ is surjective and $\im (i)=\ker(\pi)$.
	\end{definition}
	
	For simplicity, we denote $i(y)$ by $y$, which implies that $R_{E}$ restricted to $I$ is $R_I$. The equivalence of extensions of Rota--Baxter groups is defined in a~manner analogous to the standard notion of equivalence for extensions of groups and skew braces.
	
	\begin{remark}
		A direct check shows that an extension of Rota--Baxter groups induces an extension of induced skew left braces.
	\end{remark}
	
	Let $\mathcal{E}: 1 \to (I,R_I ) \stackrel{i}{\to} (E,R_E ) \stackrel{\pi}{\to} (H,R_H ) \to 1$ be a~Rota--Baxter extension of $(H, R_H )$ by $(I, R_I )$, where $I$ is an abelian group. By a~(normalised) set-theoretic section to $\mathcal{E}$, we mean a~map $s: H \rightarrow E$ such that $\pi \,s= \Id_{H}$ and $s(1)=1$. Let $\gamma:H \rightarrow \Aut(I)$ be the map defined by
	\[
	\gamma_h (y) = s(h)^{-1}y s(h).
	\]
	Note that $ \gamma$ is an anti-homomorphism, and is independent of the choice of a~set-theoretic section. We call $\gamma$ the \textit{associated action} of the extension $\mathcal{E}$. It is not difficult to see that equivalent extensions of Rota--Baxter groups have the identical associated actions. For each $a\in E$, there exists unique $h\in H$ and $y\in I$ such that $a=s(h)y$. Hence, we can write
	\begin{equation}
		R_E \big(s(h)\big)=s \big(R_H (h)\big)y_h
	\end{equation}
	for some unique $y_h\in I$. Consider the maps $\tau\colon H\times H\to I$ and $r\colon H\to I$ given by
	\[
	\tau(h_1,h_2) = s(h_1h_2)^{-1}s(h_1)s(h_2)
	\quad \text{and} \quad
	r(h) = y_h
	\]
	for $h_1, h_2, h \in H$.
	
	\begin{definition}
		Let $(H, R_H )$ be a~Rota--Baxter group, let $I$ be an abelian group, and let $R_I: I \rightarrow I$ be a~group homomorphism. We say that $(I, R_I )$ is a~right $(H, R_H )$-module if $I$ is a~right $H$-module by an action $\gamma:H \rightarrow \Aut(I)$ and the condition
		\[
		\gamma_{R_H(h)} \big(R_I(z) \big)= R_I\big(\gamma_{ h R_H( h)}(z  +R_I(z))  -\gamma_{R_H(h)}(R_I(z))\big)
		\]
		holds for all $h \in H$ and $z \in I$.
	\end{definition}
	
	Let $(I, R_I )$ be an $(H, R_H)$-module by an action $\gamma$. Let $C^n (H, I)$ be the set of all maps $H^n \to I$ that vanish on all degenerate tuples. Define:
	\begin{gather*}
		TC^1_{RB}(H, I) = C^1 (H, I), \quad
		TC^2_{RB}(H, I) = C^2 (H, I) \oplus C^{1}(H, I), \\
		TC^3_{RB}(H, I)= C^3 (H, I) \oplus C^{2}(H, I).
	\end{gather*}
	Define $\partial^1_{RB}: TC^1_{RB}(H, I) \rightarrow TC^2_{RB}(H, I) $ by
	\[
	\partial^1_{RB}(\theta) = \big(\delta^1 (\theta),  -\Phi^1(\theta)\big),
	\]
	where $\Phi^1 (\theta)(h)=R_I \big( \gamma_{R_H (h)}(\theta(h)) \big) -\theta \big(R_H (h) \big)$ and $\delta^1$ is the standard 1-coboundary map defining the group cohomology of $H$ with coefficients in $I$ (see~\cite[Section 3.1]{MR0672956}).
	
	Similarly, define $\partial^2_{RB}: TC^2_{RB}(H, I) \rightarrow TC^3_{RB}(H, I)$ by
	\begin{align*}
		\partial^2_{RB}(f, g)=(\delta^2 f,\beta),
	\end{align*}
	where
	\begin{align*}
		\beta(h_1,h_2 )&=\partial^1 (g)(h_1,h_2 )-R_I\big(\gamma_{R_H (h_2 )}(\gamma_{h_2}(g(h_1 ))-g(h_1 ))\big)-\Phi^2 (f)(h_1,h_2 ),\\
		\partial^1 (g)(h_1,h_2 )&=g(h_2 )-g(h_1\circ_{R_H} h_2 )+\gamma_{R_H (h_2 )}(g(h_1 )),\\
		\Phi^2 (f)(h_1,h_2 )&=f(R_H (h_1 ),R_H (h_2 ))- R_I\left(\gamma_{R_H (h_1\circ_{R_H} h_2 )}(f(h_1 R_H (h_1 ),h_2 R_H (h_1 )^{-1}))\right.\\
		&\ \ \left.+\gamma_{h_2 R_H (h_1 )^{-1}}(f(h_1,R_H (h_1 )))+f(h_2,R_H (h_1 )^{-1})-f(R_H (h_1 ),R_H (h_1 )^{-1})\right),
	\end{align*}
	and $\delta^2$ is the standard 2-coboundary map defining the group cohomology of $H$ with coefficients in $I$ (see~\cite[Section 3.1]{MR0672956}). Let
	\[
	\B_{RB}^2(H, I) =  \im(\partial^1_{RB}) \quad \textrm{and}\quad \Z_{RB}^2(H, I) =  \ker(\partial^2_{RB}).
	\]
	Then we define the second cohomology of $(H, R_H )$ with coefficients in $(I, R_I )$ by
	\[
	\Ho^2_{RB}(H, I)=\Z_{RB}^2(H, I)/  \B^2_{RB}(H, I).
	\]
	With the preceding set-up, the following result holds~\cite[Theorem 4.4]{MR4644858}.
	
	\begin{theorem}\label{extnrb}
		Let $(H,R_H )$ be a~Rota--Baxter group, $(I,R_I )$ an $(H,R_H )-$module and $\Ext_\gamma(H,I)$ be the set of equivalence classes of all extensions of $(H,R_H )$ by $(I,R_I )$ inducing the action $\gamma$. Then the function $\Lambda:\Ext_{\gamma}(H, I) \rightarrow \Ho^2_{RB}(H, I)$ given by $\Lambda([\mathcal{E}])=[\tau, r]$ is a~bijection, such that
		\[
		\tau(h_1, h_2) = s(h_1  h_2)^{-1}  s(h_1)  s(h_2)
		\quad \text{and} \quad
		r(h) = s \big(R_H(h)\big)^{-1}R_E \big(s(h)\big)
		\]
		for all $h_1, h_2, h \in H$	and $s$ is a~set-theoretic section to $\mathcal{E}$.
	\end{theorem}
	
	\subsection{Abelian extensions of relative Rota--Baxter groups}\label{cohomology RRB}
	Finally, we recall essential results on extension theory of relative Rota--Baxter groups from~\cite{MR4819014}.
	
	\begin{definition}
		A relative Rota--Baxter group is a~quadruple $(H, G, \phi, R)$, where $H$ and $G$ are groups, $\phi: G \rightarrow \Aut(H)$ a~group homomorphism and $R: H \rightarrow G$ is a~map satisfying the condition
		\[
		R(h_1) R(h_2)=R\big(h_1 \phi_{R(h_1)}(h_2) \big),
		\]
		for all $h_1, h_2 \in H$. The map $R$ is referred to as the relative Rota--Baxter operator on $H$. The relative Rota--Baxter group $(H, G, \phi, R)$ is called trivial if $\phi: G \to \Aut(H)$ is the trivial homomorphism.
	\end{definition}
	
	\begin{remark}\label{RRB is RB}
		Let $\phi: G \to \Aut(G)$ be the adjoint action, that is, $\phi_y (x) = yxy^{-1}$ for $x, y \in G$. Then the relative Rota--Baxter group $(G, G,\phi, R)$ is simply the Rota--Baxter group $(G, R)$.
	\end{remark}
	
	The following result connects relative Rota--Baxter groups to skew left braces~\cite[Proposition 3.6]{MR4857550}.
	
	\begin{proposition}
		If $(H, G, \phi, R)$ is a~relative Rota--Baxter group, then $(H, \cdot, \circ_R )$ is a~skew left brace, where $\cdot$ denotes the group operation on $H$ and $\circ_R$ is defined by
		\[
		h_1 \circ_R h_2 = h_1\cdot \phi_{R(h_1)}(h_2)
		\]
		for $h_1, h_2 \in H$.
	\end{proposition}
	
	In the converse direction, it is known from~\cite[Theorem 3.10]{MR4857550} that if $(H,\cdot,\circ)$ is a~skew left brace, then $(H^{(\cdot)}, H^{(\circ)}, \lambda^H, \Id_H)$ is a~relative Rota--Baxter group.
	
	\begin{definition}
		Let $(H, G, \phi, R)$ and $(K, L, \varphi, S)$ be two relative Rota--Baxter groups. A pair $(f_1, f_2 ): (H, G, \phi, R) \to (K, L, \varphi, S)$ is called a~homomorphism of relative Rota--Baxter groups if $f_1: H \rightarrow K$ and $f_2: G \rightarrow L$ are group homomorphisms such that
		\begin{equation}
			f_2 \; R = S \; f_1 \quad \textrm{and} \quad f_1 \; \phi_g = \varphi_{f_2 (g)}\; f_1
		\end{equation}
		for all $g \in G$.
	\end{definition}
	
	We write $\textbf{1}$ to denote the trivial relative Rota--Baxter group for which both the underlying groups and the maps are trivial.
	
	\begin{definition}
		Let $(K,L, \alpha,S )$ and $(A,B, \beta, T)$ be relative Rota--Baxter groups. An extension of $(A,B, \beta, T)$ by $(K,L, \alpha,S )$ is a~relative Rota--Baxter group $(H,G, \phi, R)$ that fits into the sequence
		\[
		\mathcal{E} : \quad \quad  {\bf 1} \longrightarrow (K,L, \alpha,S ) \stackrel{(i_1, i_2)}{\longrightarrow}  (H,G, \phi, R) \stackrel{(\pi_1, \pi_2)}{\longrightarrow} (A,B, \beta, T) \longrightarrow {\bf 1},
		\]
		where $(i_1, i_2 )$ and $(\pi_1, \pi_2 )$ are homomorphisms of relative Rota--Baxter groups such that $(i_1, i_2 )$ is an embedding, $(\pi_1, \pi_2 )$ is an epimorphism of relative Rota--Baxter groups and $(\im(i_1 ), \im(i_2 ), \phi|, R|)= (\Ker(\pi_1 ), \Ker(\pi_2 ), \phi|, R|)$.
		
		We say that $\mathcal{E}$ is an abelian extension if $K$ and $L$ are abelian groups and the relative Rota--Baxter group $(K,L, \alpha,S )$ is trivial.
	\end{definition}
	
	Equivalence of extensions of relative Rota--Baxter groups is defined in the same way as for groups, skew braces and Rota--Baxter groups. Throughout the immediate discussion, $\mathcal{E}$ denotes the abelian extension
	\[
	{\bf 1} \longrightarrow (K,L, \alpha,S ) \stackrel{(i_1, i_2)}{\longrightarrow}  (H,G, \phi, R) \stackrel{(\pi_1, \pi_2)}{\longrightarrow} (A,B, \beta, T) \longrightarrow {\bf 1}
	\]
	of relative Rota--Baxter groups and $(s_H, s_G )$ denotes a~set-theoretic section to $\mathcal{E}$.
	
	\begin{proposition}
		[\!{\cite[p.11]
			{MR4819014}}] \label{f rho chi eqn}
		Let $\mathcal{E}$ be an abelian extension of relative Rota--Baxter groups. Let $a \in A$, $b \in B$, $k \in K$, and $l \in L$. Then the following statements hold:
		\begin{enumerate}
			\item The action $\phi$ is characterised by the equation
			\begin{equation}
				\phi_{s_G (b)l} \big(s_H (a)k \big) = s_H \big(\beta{_b (a)} \big) \, \rho(a,b) \, \phi_{s_G (b)} \big(f(l,a)k\big),
			\end{equation}
			where $f: L \times A \rightarrow K$ is given by
			\[
			f(l,a) = s_H(a)^{-1} \phi_l(s_H(a))
			\]
			and $\rho: A \times B \rightarrow K$ is given by
			\[
			\rho(a,b) = (s_H(\beta_b(a)))^{-1}\phi_{s_G(b)}(s_H(a)).
			\]
			
			\item The relative Rota--Baxter operator $R$ is expressed as
			\begin{equation}
				R \big(s_H (a)k \big) = s_G (T(a)) \, \chi(a) \, S \big(\phi^{-1}_{s_G (T(a))}(k) \big),
			\end{equation}
			where $\chi: A \rightarrow K$ is given by $\chi(a)=s_G (T(a))^{-1}R(s_H (a))$.
		\end{enumerate}
		
	\end{proposition}
	
	\begin{lemma}[\!\!{\cite[Lemma 3.9]{MR4819014}}] \label{properties of f}
		Let $\mathcal{E}$ be an abelian extension of relative Rota--Baxter groups. Then the following assertions hold:
		\begin{enumerate}
			\item The map $f: L \times A \longrightarrow K$ is independent of the choice of the section $s_H$.
			\item $f(l_1 l_2,a) = f( l_1, a)f( l_2,a)$ for all $l_1, l_2 \in L$ and $a \in A$.
			\item $f(l, a_1 a_2 ) = \mu_{ a_2}(f(l,a_1 ))f(l, a_2 )$ for all $l \in L$ and $a_1, a_2 \in A$.
		\end{enumerate}
	\end{lemma}
	
	\begin{proposition}[\!\!{\cite[Proposition 3.7]{MR4819014}}] \label{construction of actions}
		Let $\mathcal{E}$ be an abelian extension of relative Rota--Baxter groups. Then the following assertions hold:
		\begin{enumerate}
			\item The map $\nu: B \rightarrow \Aut(K)$ defined by
			\begin{equation}\label{nuact}
				\nu_b (k)= \phi_{s_G (b)}(k)
			\end{equation}
			for $b \in B$ and $k \in K$, is a~homomorphism of groups.
			\item The $\mu: A \rightarrow \Aut(K)$ defined by
			\begin{equation}\label{muact}
				\mu_a (k)=s_H (a)^{-1}\, k\, s_H (a)
			\end{equation}
			for $a \in A$ and $ k \in K$, is an anti-homomorphism of groups.
			\item The map $\sigma: B \rightarrow \Aut(L)$ defined by
			\begin{equation}\label{sigmaact}
				\sigma_b (l)=s_G (b)^{-1}\,l\, s_G (b)
			\end{equation}
			for $b \in B$ and $ l \in K$, is an anti-homomorphism of groups.
			
			\item[] 	Further, all the maps are independent of the choice of a~section to $\mathcal{E}$.
			\item The map $\tau_1: A \times A \rightarrow K$ given by
			\begin{equation}\label{mucocycle}
				\tau_1 (a_1, a_2 )= s_H (a_1 a_2 )^{-1}s_H (a_1 )s_H (a_2 )
			\end{equation}
			for $a_1, a_2 \in A$ is a~group 2-cocycle with respect to the action $\mu$.
			\item The map $\tau_2: B \times B \rightarrow L$ given by
			\begin{equation}\label{sigmacocycle}
				\tau_2 (b_1, b_2 )= s_G (b_1 b_2 )^{-1}s_G (b_1 )s_G (b_2 )
			\end{equation}
			for	$b_1, b_2 \in B$ is a~group 2-cocycle with respect to the action $\sigma$.
		\end{enumerate}
	\end{proposition}
	
	By examining the relationships between $\nu, \mu, \sigma, f$ and by their properties outlined in Lemma~\ref{properties of f} and Proposition~\ref{construction of actions}, we are led to the following definition of a~module over a~relative Rota--Baxter group~\cite[Definition 3.12]{MR4819014}.
	
	\begin{definition}
		A module over a~relative Rota--Baxter group $(A, B,\beta, T)$ is a~trivial relative Rota--Baxter group $(K, L, \alpha, S)$ such that there exists a~quadruple $(\nu, \mu, \sigma, f)$ of maps (called an action) satisfying the following conditions:
		\begin{enumerate}
			\item The group $K$ is a~left $B$-module and a~right $A$-module with respect to the actions $\nu:B \to \Aut(K) $ and $\mu: A \to \Aut(K)$, respectively.
			\item The group $L$ is a~right $B$-module with respect to the action $\sigma: B \to \Aut(L)$.
			\item The map $f: L \times A \to K$ has the property that $f(-,a): L \rightarrow K$ is a~homomorphism for all $a \in A$ and $f(l,-): A \rightarrow K$ is a~derivation with respect to the action $\mu$ for all $l \in L$.
			\item $S\big(\nu^{-1}_{T(a)}(\mu_{a}(k)) \,\nu^{-1}_{T( a)}(f(S(k), a))\big)=\;\sigma_{T(a)}(S(k))$ for all $ a~\in A$ and $k \in K$.
			\item $\nu_b (\mu_{a}(k))= \; \mu_{ \beta{_{b}(a)}}( \nu_{b} (k))$ for all $a \in A$, $b \in B$ and $k \in K$.
		\end{enumerate}
	\end{definition}
	
	Let $G$ be a~group and $M$ an abelian group. Let $C^n (G, M)$ denote the group of maps $G^n \to M$ that vanish on degenerate tuples. Similarly, for another group $H$, let $C(G \times H, M)$ denote the group of maps $G \times H \to M$ that vanish on degenerate tuples.
	
	Let $\mathcal{K} = (K, L, \alpha, S)$ be an $\mathcal{A} = (A, B, \beta, T)$-module with respect to the action $(\nu, \mu, \sigma, f)$. We set
	\begin{eqnarray*}
		\C^{1}_{RRB}(\mathcal{A}, \mathcal{K}) &:=& C^1 (A, K) \oplus C^1 (B,L)~\textrm{and}\\
		\C^{2}_{RRB}(\mathcal{A}, \mathcal{K}) &:=& C^2 (A, K) \oplus C^2 (B,L) \oplus C(A \times B, K) \oplus C(A,L).
	\end{eqnarray*}
	Let $\Z^2_{RRB}(\mathcal{A}, \mathcal{K})$ be the subgroup of $\C^{2}_{RRB}(\mathcal{A}, \mathcal{K})$ consisting of the elements $(\tau_1, \tau_2, \rho, \chi)$ that satisfy the conditions
	\begin{eqnarray}
		\tau_1 (a_2, a_3 ) \tau_1 ( a_1, a_2 a_3 ) &=& \tau_1 (a_1 a_2, a_3 ) \mu_{a_3}(\tau_1 (a_1, a_2 )),\label{cocycle1}\\
		\tau_2 (b_2, b_3 ) \tau_2 ( b_1, b_2 b_3 ) &=& \tau_2 (b_1 b_2, b_3 ) \sigma_{b_3}(\tau_2 (b_1, b_2 )),\label{cocycle2}\\
		\rho(\beta_{b_2}(a_1 ), b_1 ) \,\nu_{b_1}(\rho(a_1,b_2 )) &=& \rho(a_1, b_1 b_2 ) \, \nu_{b_1 b_2}( f(\tau_2 (b_1, b_2 ), a_1 )),\label{cocycle3}\\
		\rho(a_1 a_2, b_1 ) \, \nu_{b_1}(\tau_1 (a_1, a_2 )) &=& \mu_{ \beta_{b_1}(a_2 )}(\rho(a_1,b_1 ))\; \rho(a_2, b_1 ) \; \tau_1 (\beta_{b_1}(a_1 ), \beta_{b_1}(a_2 )), \\ \nonumber
		\tau_2 (T(a_1 ), T(a_2 ) ) \delta^1_{\sigma}(\chi)(a_1, a_2 ) &= & S \big(\nu^{-1}_{T(a_1 \circ_T a_2 )}\big(\rho(a_2, T(a_1 )) \, \tau_1 (a_1,\beta_{T(a_1 )}(a_2 )) \nonumber \\
		&& \quad \nu_{T(a_1 )}(f(\chi(a_1 ), a_2 )) \big)\big)\label{cocycle5}
	\end{eqnarray}
	for all $a_1, a_2, a_3 \in A$ and $b_1, b_2, b_3 \in B$, where
	\[
	\delta^1_{\sigma}(\chi)(a_1, a_2)= \chi(a_2) \chi(a_1 \circ_T a_2)^{-1} \sigma_{T(a_2)}(\chi(a_1)).
	\]
	
	Next, let $\B^2_{RRB}(\mathcal{A}, \mathcal{K})$ be the subgroup of $\Z^2_{RRB}(\mathcal{A}, \mathcal{K})$ formed by elements $(\tau_1, \tau_2, \rho, \chi)$ for which there exist $\kappa_1, \kappa_2 \in \C^{1}_{RRB}(\mathcal{A}, \mathcal{K})$ satisfying the conditions
	\begin{eqnarray*}
		\tau_1 (a_1, a_2 ) &=&\kappa_1 (a_1 a_2 )^{-1}\kappa_1 (a_2 ) \mu_{a_2} (\kappa_1 (a_1 )),\\
		\tau_2 (b_1, b_2 ) &=& \kappa_2 (b_1 b_2 )^{-1} \kappa_2 (b_2 ) \sigma_{b_2} (\kappa_2 (b_1 )),\\
		\rho(a_1, a_2 ) &=&\nu_b\big( f(\kappa_2 (b_1 ), a_1 )\kappa_1 (a_1 ) \big) \big(\kappa_1 (\beta_{b_1}(a_1 ))\big)^{-1},\\
		\chi(a_1 ) &=& S \big(\nu^{-1}_{T(a_1 )}(\kappa_1 (a_1 )) \big) \big( \kappa_2 (T(a_1 )) \big)^{-1}
	\end{eqnarray*}
	for all $a_1, a_2 \in A$ and $b_1, b_2 \in B$.
	
	We define the second cohomology group of $\mathcal{A}$ with coefficients in $\mathcal{K}$ by
	\[
	\Ho^2_{RRB}(\mathcal{A}, \mathcal{K})=\Z^2_{RRB}(\mathcal{A}, \mathcal{K})/  \B^2_{RRB}(\mathcal{A}, \mathcal{K}).
	\]
	
	Consider a~2-cocycle $(\tau_1, \tau_2, \rho, \chi) \in \Ker(\partial_{RRB}^2)$. Define $H=A\times_{\tau_1} K$ and $G= B \times_{\tau_2} L$ to be the group extensions of $A$ by $K$ and $B$ by $L$ associated to the group 2-cocycles $\tau_1$ and $\tau_2$, respectively. Further, define $\phi: G \rightarrow \Aut(H)$ by
	\begin{align}
		\phi_{(b,l)}(a,k)=\big(\beta{_{b}(a)}, \rho(a,b) \,\nu_b (f(l,a)k)\big)
	\end{align}
	and $R: H \rightarrow G$ by
	\begin{align}
		R(a,k)=\big(T(a), \chi(a)\,S(\nu^{-1}_{T(a)}(k))\big)
	\end{align}
	for all $a \in A$, $b \in B$, $k \in K$ and $l \in L$. Using the fact that $(\tau_1, \tau_2, \rho, \chi) \in \Ker(\partial_{RRB}^2 )$, it follows that $(H, G, \phi, R)$ is a~relative Rota--Baxter group and is an extension of $(A,B, \beta, T)$ by $(K,L, \alpha, S)$ denoted by
	\[
	\mathcal{E}(\tau_1, \tau_2, \rho, \chi) : \quad \quad {\bf 1} \to (K,L, \alpha,S ) \stackrel{(i_1, i_2)}{\longrightarrow}  (H, G, \phi, R) \stackrel{(\pi_1, \pi_2)}{\longrightarrow} (A,B, \beta, T) \to {\bf 1}.
	\]
	
	Let $\Ext_{(\nu, \mu, \sigma, f)}(\mathcal{A}, \mathcal{K})$ denote the set of equivalence classes of extensions of $\mathcal{A}$ by $\mathcal{K}$ for which the associated action is $(\nu, \mu, \sigma, f)$. Then we have the following result~\cite[Theorem~3.18]{MR4819014}.
	
	\begin{theorem}\label{ext and cohom bijection}
		Let $\mathcal{A}= (A,B, \beta, T)$ be a~relative Rota--Baxter group and $\mathcal{K}= (K,L,\alpha,S )$ be a~trivial relative Rota--Baxter group, where $K$ and $L$ are abelian groups. Let $(\nu, \mu, \sigma, f)$ be the quadruple of actions turning $\mathcal{K}$ into an $\mathcal{A}$-module. Then there is a~bijective function $\Gamma : \Ho^2_{RRB}(\mathcal{A},\mathcal{K})\to \Ext_{(\nu,\mu,\sigma,f)}(\mathcal{A},\mathcal{K})$ given by
		\[
		\Gamma \big([\tau_1, \tau_2, \rho, \chi] \big)=[\mathcal{E}(\tau_1, \tau_2, \rho, \chi)].
		\]
	\end{theorem}
	
	\section[From relative RB groups to RB groups on semi-direct products]{From relative Rota--Baxter groups to Rota--Baxter groups on semi-direct products}\label{section RRB to RRB on SDP}
	Recall that, the semi-direct product $H\rtimes_{\phi} G$ of groups $G$ and $H$ under the (left) action $\phi\colon G\rightarrow \Aut(H)$ is the set
	$H \times G$ equipped with the multiplication
	\[
	(h_1, g_1) (h_2, g_2) = \big(h_1 \phi_{g_1}(h_2), g_1 g_2 \big)
	\]
	for all $h_1, h_2 \in H$ and $g_1, g_2 \in G$.
	
	The following result shows that each relative Rota--Baxter group defines a~Rota--Baxter group on the natural semi-direct product~\cite[Proposition 3.1]{MR4868162}.
	
	\begin{proposition}\label{RR-R}
		Let $(H, G, \phi, R)$ be a~relative Rota--Baxter group. Then, the function $\widetilde{R} \colon H\rtimes_{\phi} G \to H\rtimes_{\phi} G$ given by
		\[
		\widetilde{R}\big((h, g)\big) = \big(1, g^{-1} R(h)\big)
		\]
		for all $h \in H$ and $g \in G$, is a~Rota--Baxter operator on the semi-direct product $H\rtimes_{\phi} G$.
	\end{proposition}
	
	We observe that the preceding construction is, in fact, functorial.
	
	\begin{proposition}\label{RRB to RB functor}
		Let $\mathcal{RRB}$ and $\mathcal{RB}$ denote the categories of relative Rota--Baxter groups and Rota--Baxter groups, respectively. Then the association $\mathcal{F}: \mathcal{RRB} \to \mathcal{RB}$ given by $\mathcal{F}\big((H, G, \phi, R) \big) = \big(H\rtimes_{\phi} G, \widetilde{R} \big)$ is a~covariant functor.
	\end{proposition}
	
	\begin{proof}
		Let $(f_1, f_2 ): (H,G,\phi,R) \to (K,L,\alpha,S)$ be a~homomorphism of relative Rota--Baxter groups. Define $\widetilde{f}:H\rtimes_{\phi} G \to K\rtimes_{\alpha} L$ by
		$\widetilde{f}(h,g)=(f_1 (h),f_2 (g))$ for all $h \in H$ and $g \in G$. We claim that $\widetilde{f}:(H\rtimes_{\phi} G, \widetilde{R})\to (K\rtimes_{\alpha}L, \widetilde{S})$ is a~homomorphism of Rota--Baxter groups. Indeed, for $g, g' \in G$ and $h, h' \in H$, we have
		\begin{eqnarray*}
			\widetilde{f}\big((h,g)(h',g')\big)&=&\widetilde{f}\big((h\phi_g (h'),gg')\big)\\
			&=&\big(f_1 (h\phi_g (h')),f_2 (gg')\big)\\
			&=& \big(f_1 (h)f_1 (\phi_g (h')),f_2 (gg') \big)\\
			&=& \big(f_1 (h)\alpha_{f_2 (g)}(f_1 (h')),f_2 (gg') \big),\\
			&& \text{since $(f_1,f_2 )$ is a~homomorphism of relative Rota--Baxter groups}\\
			&=&\big(f_1 (h),f_2 (g)\big)\big(f_1 (h'),f_2 (g')\big)\\
			&=& \widetilde{f} \big((h,g)\big) \widetilde{f} \big((h',g')\big)
		\end{eqnarray*}
		and
		\begin{eqnarray*}
			\widetilde{f}\, \widetilde{R}\big((h,g)\big)&=&\widetilde{f} \big((1,g^{-1}R(h)) \big)\\
			&=& \big(1,f_2 (g^{-1}R(h) \big)\\
			&=& \big(1,f_2 (g)^{-1}f_2 R(h) \big)\\
			&=& \big(1,f_2 (g)^{-1}S f_1 (h)\big),\\
			&& \text{since $(f_1,f_2 )$ is a~homomorphism of relative Rota--Baxter groups}\\
			&=&\widetilde{S} \big(f_1 (h),f_2 (g) \big)\\
			&=&\widetilde{S} \, \widetilde{f} \big((h,g)\big).
		\end{eqnarray*}
		It is now immediate to see that $\mathcal{F}: \mathcal{RRB} \to \mathcal{RB}$ is a~functor.
	\end{proof}
	
	In fact, $\mathcal{F}$ maps equivalent extensions to equivalent extensions.
	
	\begin{corollary}\label{RRB2rbext}
		Let $\mathcal{A}= (A,B, \beta, T)$ be a~relative Rota--Baxter group, $\mathcal{K}= (K,L,\alpha,S )$ a~trivial relative Rota--Baxter group such that $K$ and $L$ are abelian groups, and
		\[
		\mathcal{E} : \quad {\bf 1} \longrightarrow (K,L,\alpha,S ) \stackrel{(i_1, i_2)}{\longrightarrow}  (H,G, \phi, R) \stackrel{(\pi_1, \pi_2)}{\longrightarrow} (A,B, \beta, T) \longrightarrow {\bf 1}
		\]
		an extension of relative Rota--Baxter groups with associated action $(\nu, \mu, \sigma, f)$. Then
		\begin{equation*}
			\widetilde{\mathcal{E}}: \quad \textbf{1} \longrightarrow (K\rtimes_{\alpha}L, \widetilde{S}) \stackrel{\widetilde{i}}{\longrightarrow} (H\rtimes_{\phi} G, \widetilde{R}) \stackrel{\widetilde{\pi}}{\longrightarrow} (A\rtimes_{\beta} B, \widetilde{T}) \longrightarrow \textbf{1}
		\end{equation*}
		is an extension of Rota--Baxter groups with associated action $\gamma$ given by
		\begin{equation}\label{gamma action rb}
			\gamma_{(a,b)}(k,l)=\big(\nu^{-1}_b (\mu_a (k)f(l,a)),\, \sigma_b (l) \big)
		\end{equation}
		for all $a\in A$, $b\in B$, $k\in K$ and $l\in L$. Further, there is a~map
		\[
		\Pi: \Ext_{(\nu, \mu, \sigma, f)}(\mathcal{A}, \mathcal{K}) \to \Ext_{\gamma}(A\rtimes_\beta B, K\times L)
		\quad \text{given by} \quad
		\Pi\big([\mathcal{E}] \big)= [\widetilde{\mathcal{E}}].
		\]
	\end{corollary}
	
	\begin{proof}
		Let $(s_H, s_G )$ be a~set-theoretic section to $\mathcal{E}$ inducing the action $(\nu, \mu, \sigma, f)$. Then the map $s((a,b))=(s_H (a),s_G (b))$ is a~set-theoretic section to $\widetilde{\mathcal{E}}$. Let $\gamma$ be the action associated to the extension $\mathcal{E}$. For $a\in A$, $b\in B$, $k\in K$ and $l\in L$, we see that
		\begin{eqnarray*}
			\gamma_{(a,b)}(k,l)&=&s(a,b)^{-1}(k,l)s(a,b)\\
			&=& \big(s_H (a),~s_G (b)\big)^{-1}(k,\,l) \big(s_H (a),~s_G (b) \big)\\
			&=&\big((\phi_{s_G (b)^{-1}}(s_H (a)))^{-1},~s_G (b)^{-1} \big) \big(k\,\phi_l (s_H (a)),~ls_G (b)\big)\\
			&=&\big((\phi_{s_G (b)^{-1}}(s_H (a)))^{-1}\phi_{s_G (b)^{-1}}(k \,\phi_l (s_H (a))),~s_G (b)^{-1}ls_G (b)\big)\\
			&=&\big(\phi_{s_G (b)^{-1}}\big(s_H (a)^{-1}k\,\phi_l (s_H (a))\big),~\sigma_b (l)\big),~\text{by~\eqref{sigmaact}}\\
			&=&\big(\phi_{s_G (b)^{-1}}\big(s_H (a)^{-1}k\,s_H (a)s_H (a)^{-1}\phi_l (s_H (a))\big),~\sigma_b (l)\big)\\
			&=&\big(\phi_{s_G (b)^{-1}}\big(\mu_a (k)f(l,a)\big),~\sigma_b (l)\big),~\text{by Lemma~\eqref{properties of f} and~\eqref{muact}}\\
			&=& \big(\nu^{-1}_b (\mu_a (k)f(l,a)),~\sigma_b (l) \big),~\text{by~\eqref{nuact}.}
		\end{eqnarray*}
		Hence, the associated action $\gamma$ of the induced extension $\widetilde{\mathcal{E}}$ is completely determined by the action $(\nu,\mu,\sigma,f)$ of the extension $\mathcal{E}$. A direct check shows that equivalent extensions of relative Rota--Baxter groups map to equivalent extensions of Rota--Baxter groups. By~\cite[Proposition 3.17]{MR4819014}, equivalent extensions of relative Rota--Baxter groups induce identical associated actions, and the result follows.
	\end{proof}
	
	Let $\mathcal{A}= (A,B, \beta, T)$ and $\mathcal{K}= (K,L,\alpha,S )$ be as in Corollary~\ref{RRB2rbext}. Define the map $\tau\colon (A\rtimes_\beta B)\times (A\rtimes_\beta B)\to K\times L$ by
	\[
	\tau \big((a_1,b_1),(a_2,b_2)\big)=s\big((a_1,b_1)(a_2,b_2)\big)^{-1}s \big((a_1,b_1)\big) s \big((a_2,b_2)\big)
	\]
	for all $a_1, a_2 \in A$ and $b_1, b_2 \in B$. Then, we have
	\begin{align}
		\tau&\big((a_1,b_1 ),(a_2,b_2 )\big) \nonumber\\
		&= s\big((a_1\beta_{b_1}(a_2 ), \,b_1 b_2 )\big)^{-1} \big(s_H (a_1 ), \,s_G (b_1 ) \big) \big(s_H (a_2 ), \,s_G (b_2 ) \big)\nonumber\\
		&= \left(s_H (a_1\beta_{b_1}(a_2)), \, s_G (b_1 b_2 )\right)^{-1} \left(s_H (a_1) \phi_{s_G (b_1)}(s_H (a_2 )), \, s_G (b_1) s_G (b_2)\right) \nonumber\\
		&= \left(\phi_{s_G (b_1 b_2)^{-1}} (s_H (a_1 \beta_{b_1}(a_1 )))^{-1}, \, s_G (b_1 b_2)^{-1} \right) \left( s_H (a_1) \phi_{s_G (b_1)} (s_H (a_2)), \, s_G (b_1) s_G (b_2)\right) \nonumber\\
		&= \left(\phi_{s_G (b_1 b_2)^{-1}} (s_H (a_1 \beta_{b_1}(a_1)))^{-1} \phi_{s_G (b_1 b_2)^{-1}} (s_H (a_1) \phi_{s_G (b_1)}(s_H (a_2))), \right. \nonumber\\
		&\qquad \left. s_G (b_1 b_2)^{-1} s_G (b_1) s_G (b_2) \right) \nonumber\\
		&=\big(\phi_{s_G (b_1 b_2 )^{-1}}\big((s_H (a_1\beta_{b_1}(a_1 )))^{-1}s_H (a_1 )\phi_{s_G (b_1 )}(s_H (a_2 ))\big), \,\tau_2 (b_1,b_2 )\big), \text{using~\eqref{sigmacocycle}}\nonumber\\
		&=\big(\phi_{s_G (b_1 b_2 )^{-1}}\big(\tau_1 (a_1,\beta_{b_1}(a_2 ))s_H (\beta_{b_1}(a_2 ))^{-1}\phi_{s_G (b_1 )}(s_H (a_2 ))\big), \,\tau_2 (b_1,b_2 )\big),~\text{using~\eqref{mucocycle}}\nonumber\\
		&=\big(\phi_{s_G (b_1 b_2 )^{-1}}\big(\tau_1 (a_1,\beta_{b_1}(a_2 ))\rho(a_2,b_1 )\big), \,\tau_2 (b_1,b_2 )\big), \text{using Proposition~\eqref{f rho chi eqn}(1)}\nonumber\\
		&=\left(\nu^{-1}_{b_1 b_2}\big(\tau_1 (a_1,\beta_{b_1}(a_2 ))\rho(a_2,b_1 )\big), \,\tau_2 (b_1,b_2 )\right). \label{tau 1 prime}
	\end{align}
	Now, define the map $r\colon A\rtimes_\beta B\to K\times L$ by
	\[
	r(a,b)=s \big(\widetilde{T}(a,b)\big)^{-1} \, \widetilde{R} \big(s(a,b)\big)
	\]
	for all $a \in A$ and $b \in B$. Then, we have
	\begin{eqnarray}
		r (a,b) &=&s \big(1,b^{-1}T(a) \big)^{-1}\, \widetilde{R} \big(s_H (a),s_G (b)\big)\nonumber\\
		&=&\big(1,s_G (b^{-1}T(a))\big)^{-1}\big(1,s_G (b)^{-1}R(s_H (a))\big)\nonumber\\
		&=&\big(1,s_G (b^{-1}T(a))^{-1}\big)\big(1,s_G (b)^{-1}R(s_H (a))\big)\nonumber\\
		&=&\big(1,s_G (b^{-1}T(a))^{-1}s_G (b)^{-1}R(s_H (a))\big)\nonumber\\
		&=&\big(1,\tau_2 (b^{-1},T(a))s_G (T(a))^{-1}s_G (b^{-1})^{-1}s_G (b)^{-1}R(s_H (a))\big), \text{using~\eqref{sigmacocycle}}\nonumber\\
		&=&\big(1,\tau_2 (b^{-1},T(a))s_G (T(a))^{-1}\tau_2 (b,b^{-1})^{-1}s_G (T(a))\chi(a)\big),\nonumber\\
		&& \text{using Proposition~\eqref{f rho chi eqn}(2) and~\eqref{sigmacocycle}}\nonumber\\
		&=&\big(1,\tau_2 (b^{-1},T(a))\sigma_{T(a)}(\tau_2 (b,b^{-1})^{-1})\chi(a)\big), \text{using~\eqref{sigmaact}}\nonumber\\
		&=&\big(1,\tau_2 (b,b^{-1}T(a))^{-1}\chi(a)\big), \text{using~\eqref{cocycle2}.\label{g cocycle term}}
	\end{eqnarray}
	Next, we explore the relationship between $\Ho^2_{RRB}(\mathcal{A}, \mathcal{K})$ and $\Ho^2_{RB}(A\rtimes_\beta B, K\times L)$ (see Subsections~\ref{RB cohomology} and~\ref{cohomology RRB} for the constructions). Let $\Gamma: \Ho^2_{RRB}(\mathcal{A}, \mathcal{K}) \to \Ext_{(\nu, \mu, \sigma, f)}(\mathcal{A}, \mathcal{K})$ and $\Lambda:\Ext_{\gamma}(A\rtimes_\beta B, K\times L) \to \Ho^2_{RB}(A\rtimes_\beta B, K\times L)$ be the bijections defined in Theorems~\ref{ext and cohom bijection} and~\ref{extnrb}, respectively. Let $\Pi: \Ext_{(\nu, \mu, \sigma, f)}(\mathcal{A}, \mathcal{K}) \to \Ext_{\gamma}(A\rtimes_\beta B, K\times L)$ be the map defined in Corollary~\ref{RRB2rbext}. Then we have the following result.
	
	\begin{proposition}\label{cohom RRB to cohom RB}
		Let $\mathcal{A}= (A,B, \beta, T)$ be a~relative Rota--Baxter group and assume that $\mathcal{K} = (K, L, \alpha, S)$ is a~module over $\mathcal{A}$ with respect to the action $(\nu, \mu,\sigma, f)$. Then the map
		\[
		\Omega_{RB}:=\Lambda \Pi \Gamma: \Ho^2_{RRB}(\mathcal{A}, \mathcal{K}) \to \Ho^2_{RB}(A\rtimes_\beta B, K\times L)
		\]
		is a~homomorphism of groups.
	\end{proposition}
	
	\begin{proof}
		The map $\Omega_{RB}$ is given by $\Omega_{RB}\big( [\tau_1, \tau_2, \rho, \chi]\big)= [\tau^{(\tau_1, \tau_2, \rho, \chi)}, \, r^{(\tau_1, \tau_2, \rho, \chi)}]$, where
		\begin{eqnarray*}
			\tau^{(\tau_1, \tau_2, \rho, \chi)} \big((a_1,b_1 ),(a_2,b_2 )\big)&=&\big(\nu^{-1}_{b_1 b_2}\big(\tau_1 (a_1,\beta_{b_1}(a_2 ))\rho(a_2,b_1 )\big),\tau_2 (b_1,b_2 )\big),\\
			r^{(\tau_1, \tau_2, \rho, \chi)}(a,b)&=&\big(1,\tau_2 (b,b^{-1}T(a))^{-1}\chi(a)\big)
		\end{eqnarray*}
		for all $a,a_1,a_2\in A$ and $b,b_1,b_2\in B$. Let $[\tau_1, \tau_2, \rho, \chi]$ and $[\tau^\prime_1, \tau^\prime_2, \rho^\prime, \chi^\prime]$ be elements of $\Ho^2_{RRB}(\mathcal{A}, \mathcal{K})$. Then, we have
		\begin{eqnarray*}
			\Omega_{RB} \big( [\tau_1, \tau_2, \rho, \chi] \, [\tau^\prime_1, \tau^\prime_2, \rho^\prime, \chi^\prime] \big) &=& \Omega_{RB} \big( [\tau_1\tau^\prime_1, \tau_2\tau^\prime_2, \rho \rho^\prime, \chi\chi^\prime]\big)\\
			&=& \big[\tau^{(\tau_1\tau^\prime_1, \tau_2\tau^\prime_2, \rho \rho^\prime, \chi\chi^\prime)},\, r^{(\tau_1\tau^\prime_1, \tau_2\tau^\prime_2, \rho \rho^\prime, \chi\chi^\prime)} \big].
		\end{eqnarray*}
		It is easy to see that
		\[
		\tau^{(\tau_1\tau^\prime_1, \tau_2\tau^\prime_2, \rho \rho^\prime, \chi\chi^\prime)}=	\tau^{(\tau_1, \tau_2, \rho, \chi)} \,	\tau^{(\tau^\prime_1, \tau^\prime_2, \rho^\prime, \chi^\prime)} \quad \textrm{and} \quad r^{(\tau_1\tau^\prime_1, \tau_2\tau^\prime_2, \rho \rho^\prime, \chi\chi^\prime)}=	r^{(\tau_1, \tau_2, \rho, \chi)}\, r^{(\tau^\prime_1, \tau^\prime_2, \rho^\prime, \chi^\prime)}.
		\]
		Hence, we obtain
		\begin{eqnarray*}
			\Omega_{RB} \big( [\tau_1, \tau_2, \rho, \chi] [\tau^\prime_1, \tau^\prime_2, \rho^\prime, \chi^\prime] \big) &=& \big[\tau^{(\tau_1, \tau_2, \rho, \chi)},r^{(\tau_1, \tau_2, \rho, \chi)} \big] \big[\tau^{(\tau^\prime_1, \tau^\prime_2, \rho^\prime, \chi^\prime)},r^{(\tau^\prime_1, \tau^\prime_2, \rho^\prime, \chi^\prime)} \big]\\
			& =& \Omega_{RB} \big( [\tau_1, \tau_2, \rho, \chi] \big) \,	\Omega_{RB} \big([\tau^\prime_1, \tau^\prime_2, \rho^\prime, \chi^\prime] \big),
		\end{eqnarray*}
		which is desired.
	\end{proof}
	
	\section{From skew left braces to skew left braces on semi-direct products}\label{section SLB to SLB on SDP}
	It is known from~\cite[Theorem 3.10]{MR4857550} that if $(H,\cdot,\circ)$ is a~skew left brace, then the quadruple $(H^{(\cdot)},H^{(\circ)},\lambda^H,\Id_H )$ is a~relative Rota--Baxter group, where
	\[
    \lambda^H: H^{(\circ)}\to \Aut(H^{(\cdot)})
    \quad \text{is defined by} \quad 
	\lambda^H_a(b)=a^{-1}\cdot(a\circ b)
	\]
	for all $a, b \in H$. Furthermore, the skew left brace induced by the relative Rota--Baxter group $(H^{(\cdot)},H^{(\circ)},\lambda^H,\Id_H )$ is the same as $(H, \cdot, \circ)$. Also, if $(H^{(\cdot)},H^{(\circ)},\lambda^H,\Id_H )$ is a~relative Rota--Baxter group, then $(H,\cdot,\circ)$ is a~skew left brace with $\lambda^H$ as its associated action.
	
	Let $(H,\cdot,\circ)$ be a~skew left brace and let
	\[
	\widetilde{H}=H^{(\cdot)} \rtimes_{\lambda^H} H^{(\circ)}
	\]
	denote the associated semi-direct product. In view of Proposition~\ref{RR-R} and Remark~\ref{RRB is RB}, there is a~skew left brace structure on $\widetilde{H}$ induced by the Rota--Baxter group $(\widetilde{H}, \widetilde{R})$, where
	$\widetilde{R}: \widetilde{H}\to \widetilde{H}$ is given by
	\begin{equation}\label{RB operator on semi-direct product}
		\widetilde{R} (h, g) = (1, g^\dagger\circ h)
	\end{equation}
	for all $h,g \in H$. The additive group operation $\bullet$ is given by
	\[
	(h_1, g_1)\bullet (h_2, g_2) = \big(h_1\cdot \lambda^H_{g_1}(h_2), \, g_1\circ g_2 \big),
	\]
	and the multiplicative group operation $\odot$ is given by
	\[
	(h_1, g_1)\odot  (h_2, g_2) = \big(h_1\circ h_2,  \, h_1\circ g_2\circ h_1^\dagger \circ g_1 \big)
	\]
	for all $g_1, g_2, h_1, h_2 \in H$.
	
	\begin{definition}
		Given a~skew left brace $(H,\cdot,\circ)$, the skew left brace $(\widetilde{H}, \bullet, \odot)$ is called the \textit{square} of $(H,\cdot,\circ)$.
	\end{definition}
	
	\begin{remark}\label{rem square arise from RBG}
		By definition, the square of a~skew left brace always arises from a~Rota--Baxter group. Further, the associated action for $(\widetilde{H}, \bullet, \odot)$ is given by
		\[
		\lambda^{\widetilde{H}}_{(h_1, g_1)}(h_2, g_2) = \widetilde{R} (h_1, g_1) \bullet (h_2, g_2) \bullet \big(\widetilde{R} (h_1, g_1)\big)^{-1}=  (1, g_1^\dagger\circ h_1) \bullet (h_2, g_2)  \bullet  (1, g_1^\dagger\circ h_1)^{-1}.
		\]
	\end{remark}
	
	\begin{example}
		Let $(H,\cdot,\cdot)$ be a~trivial skew left brace. Then, we have
		\begin{eqnarray*}
			(h_1, g_1 )\bullet (h_2, g_2 ) &=& \big(h_1\cdot h_2, \,g_1\cdot g_2 \big)~\textrm{and}\\
			(h_1, g_1 )\odot (h_2, g_2 ) &=& \big(h_1\cdot h_2, \, h_1\cdot g_2\cdot h_1^{-1} \cdot g_1 \big)\nonumber.
		\end{eqnarray*}
		Thus, the square of a~trivial skew left brace need not be a~trivial skew left brace. However, the square of a~trivial left brace is always trivial.
	\end{example}
	
	Recall the following definition from~\cite[Corollary 3.38]{MR3763907}.
	
	\begin{definition}
		Let $(H,\cdot,\circ)$ be a~skew left brace such that
		\begin{equation*}
			\lambda^H_h\lambda^H_g =\lambda^H_{\lambda^H_h (g)}\lambda^H_h
		\end{equation*}
		for all $g, h \in H$. Then $H\times H $ admits a~skew left brace structure, where
		\[
		(h,g) \cdot (h',g') = (h \cdot h', \, g \cdot g') \quad \text{and} \quad
		(h,g)\circ(h',g') = \big(h\circ h', \, g\circ\lambda^H_h (g') \big).
		\]
		This skew left brace is called the double of $(H,\cdot,\circ)$ and is denoted by $(D(H), \cdot,\circ)$.
	\end{definition}
	
	\begin{remark}
		It follows from the definition that the double of a~trivial skew left brace is always a~trivial skew left brace. If $(H,\cdot,\circ)$ is a~skew left brace, then the associated action for the double $(D(H), \cdot,\circ)$ is given by
		\begin{equation*}
			\lambda^{D(H)}_{(h,g)}(h',g')=\big(\lambda^H_h (h'), \, \lambda^H_{g\circ h}(g')\big).
		\end{equation*}
		Thus, if $\lambda^H_h$ is not an inner automorphism of $H^{(\cdot)}$ for some $h$, then $\lambda^{D(H)}_{(h,g)}$ is not an inner automorphism of $D(H)^{(\cdot)}$ for all $(h,g)$. By~\eqref{RB induced SLB associated action}, the necessary condition for a~skew left brace to be induced from a~Rota--Baxter operator is that its associated action must take values in the inner automorphism group of its additive group. Thus, if $\lambda^H_h$ is not an inner automorphism of $H^{(\cdot)}$ for some $h$, then by Remark~\ref{rem square arise from RBG}, the double $(D(H), \cdot, \circ)$ and the square $(\widetilde{H}, \bullet, \odot)$ are distinct.
	\end{remark}
	
	\begin{remark}
		By a~similar argument, the square of a~skew left brace is seen to differ from the double semi-direct product of skew left braces~\cite[Corollary 21]{MR4209715} and the \textit{twofold semi-direct products of skew left braces}~\cite[Proposition 2.1]{MR4950651}, where the latter relates to \textit{the semi-direct product of digroups}~\cite{MR4745565}.
	\end{remark}
	
	\begin{example}
		Consider the skew left brace $(\mathbb{Z}, +, \circ)$, where $n \circ m= n +(-1)^n m$. Note that $\lambda^{\mathbb{Z}}_n (m)=(-1)^{n}m$, and hence $\lambda^{\mathbb{Z}}_n$ is not an inner automorphism of $\mathbb{Z}^{(+)}$ for each $n \in \mathbb{Z}$. For the square $(\widetilde{\mathbb{Z}},\bullet,\odot)$, we have
		\begin{eqnarray*}
			(n_1,m_1 )\bullet (n_2,m_2 )&=&\big(n_1 +(-1)^{m_1}n_2,~m_1 + (-1)^{m_1} m_2\big)~\textrm{and}\\
			(n_1,m_1 )\odot(n_2,m_2 ) &=&\big(n_1 + (-1)^{n_1} n_2, n_1 + (-1)^{n_1}m_2 + (-1)^{ m_2}(m_1 -n_1 )\big).
		\end{eqnarray*}
		On the other hand, for the double $(D(\mathbb{Z}),\cdot,\circ)$, we have
		\begin{eqnarray*}
			(n_1,m_1 ) \cdot (n_2,m_2 )&=&(n_1 +n_2, m_1 +m_2 )~\textrm{and}\\
			(n_1,m_1 )\circ(n_2,m_2 )&=&\big(n_1 + (-1)^{n_1} n_2, m_1 +(-1)^{n_1 +m_1}m_2\big).
		\end{eqnarray*}
		Thus, the square and the double of the skew brace $(\mathbb{Z},+,\circ)$ are not isomorphic to each other.
	\end{example}
	
	We prove that the association $(H,\cdot,\circ) \mapsto (\widetilde{H}, \bullet, \odot)$ is functorial.
	
	\begin{proposition}\label{sb iso to isb iso}
		Let $\mathcal{SB}$ denote the category of skew left braces. Then the association $\mathcal{O}: \mathcal{SB} \to \mathcal{SB}$ given by $\mathcal{O}((H,\cdot,\circ))= (\widetilde{H}, \bullet, \odot)$ is a~covariant functor.
	\end{proposition}
	
	\begin{proof}
		Let $f\colon (H,\cdot,\circ)\to(G,\cdot,\circ)$ be a~morphism of skew left braces. Define
		\[
		\tilde{f}: (\tilde{H},\bullet,\odot)\to (\tilde{G},\bullet,\odot)
		\quad \text{by} \quad \widetilde{f}\big((h,g)\big)=\big(f(h),f(g)\big)\] 
		
		for all $g, h \in H$. We claim that $\widetilde{f}$ is a~morphism of skew left braces. Indeed, we have
		\begin{eqnarray*}
			\widetilde{f} \big((h,g)\bullet(h',g') \big)&=&\widetilde{f} \big((h \cdot \lambda^H_g (h'),~g\circ g')\big)\\
			&=& \big(f(h \cdot \lambda^H_g (h')),~f(g\circ g')\big)\\
			&=& \big(f(h) \cdot f(\lambda^H_g (h')),~f(g)\circ f(g') \big),\\
			&&\text{since $f$ is a homomorphism of skew left braces}\\
			&=& \big(f(h)\, \lambda^G_{f(g)}(f(h')),~f(g)\circ f(g') \big),\\
			&&\text{since $f$ is a homomorphism of skew left braces}\\
			&=& \big(f(h),f(g) \big)\bullet \big(f(h'),f(g') \big)\\
			&=&\widetilde{f}\big((h,g)\big) \bullet \widetilde{f}\big((h',g')\big)
		\end{eqnarray*}
		and
		\begin{eqnarray*}
			\widetilde{f} \big((h,g)\odot (h',g') \big)&=&\widetilde{f} \big((h\circ h',~h\circ g'\circ h^\dagger\circ g)\big)\\
			&=&	\big(f(h\circ h'),~f(h\circ g'\circ h^\dagger\circ g) \big)\\
			&=& \big(f(h),f(g) \big)\odot \big(f(h'),f(g') \big)\\
			&=&\widetilde{f} \big((h,g)\big)\odot {f} \big((h',g')\big),
		\end{eqnarray*}
		for $g,g', h,h' \in H$. It is now immediate to see that $\mathcal{O}: \mathcal{SB} \to \mathcal{SB}$ is a~functor.
	\end{proof}
	
	In fact, $\mathcal{O}$ maps equivalent extensions to equivalent extensions.
	
	\begin{corollary}\label{extsb-extdoublesb}
		Let $(H,\cdot,\circ)$ be a~skew left brace and $(I,\cdot)$ an abelian group viewed as a~trivial brace. Let
		\[
		\mathcal{E} : \quad  {\bf 1} \longrightarrow (I,\cdot) \stackrel{i}{\longrightarrow}  (E,\cdot,\circ) \stackrel{\pi}{\longrightarrow} (H,\cdot,\circ) \longrightarrow {\bf 1}
		\]
		be an extension of skew left braces with associated action $(\xi,\zeta,\epsilon)$ and
		\[
		\widetilde{\mathcal{E}}: \quad  {\bf 1} \longrightarrow (\widetilde{I},\cdot) \stackrel{\widetilde{i}}{\longrightarrow}  (\widetilde{E},\bullet,\odot) \stackrel{\widetilde{\pi}}{\longrightarrow} (\widetilde{H},\bullet,\odot) \longrightarrow {\bf 1}
		\]
		the induced extension of skew left braces with associated action $(\widetilde{\xi},\widetilde{\zeta},\widetilde{\epsilon})$. Then there is a~map
		$\Theta: \Ext_{(\xi,\zeta,\epsilon)}(H, I) \to \Ext_{(\widetilde{\xi},\widetilde{\zeta},\widetilde{\epsilon})}(\widetilde{H}, \widetilde{I})$ given by $\Theta\big([\mathcal{E}] \big)= [\widetilde{\mathcal{E}}]$.
	\end{corollary}
	
	\begin{proof}
		Note that, since $I$ is a~trivial brace, $\widetilde{I}= I \times I$ is again a~trivial brace. A direct check shows that if
		\[
		\mathcal{E} : \quad  {\bf 1} \longrightarrow (I,\cdot) \stackrel{i}{\longrightarrow}  (E,\cdot,\circ) \stackrel{\pi}{\longrightarrow} (H,\cdot,\circ) \longrightarrow {\bf 1}
		\]
		is an abelian extension of skew left braces, then the induced sequence
		\[
		\widetilde{\mathcal{E}}: \quad  {\bf 1} \longrightarrow (\widetilde{I},\cdot) \stackrel{\widetilde{i}}{\longrightarrow}  (\widetilde{E},\bullet,\odot) \stackrel{\widetilde{\pi}}{\longrightarrow} (\widetilde{H},\bullet,\odot) \longrightarrow {\bf 1}
		\]
		of skew left braces is also exact. Let $s$ be a~set-theoretic section to $\mathcal{E}$. Then, we have that $\widetilde{s}(h,g)=(s(h),s(g))$ is a~set-theoretic section to $\widetilde{\mathcal{E}}$. Let $(\widetilde{\xi}, \widetilde{\zeta}, \widetilde{\epsilon})$ be the associated action of $\widetilde{\mathcal{E}}$. For $a,b\in I$ and $h,g\in H$, we see that
		\begin{eqnarray}
			\widetilde{\xi}_{(h,g)}(a,b)&=&\lambda^{\widetilde{E}}_{(s(h),s(g))}(a,b)\nonumber\\
			&=& \big(1,s(g)^\dagger \circ s(h) \big) \bullet (a,b) \bullet \big(1,s(h)^\dagger\circ s(g)\big), \text{by definition of $\lambda^{\widetilde{E}}_{(s(h),s(g))}$}\nonumber\\
			&=& \big(1,s(g)^\dagger \circ s(h)\big) \bullet \big(a,b\circ s(h)^\dagger \circ s(g)\big)\nonumber\\
			&=& \big(\lambda^E_{s(g)^\dagger}(\lambda^E_{s(h)}(a)),~s(g)^\dagger \circ s(h)\circ b\circ s(h)^\dagger \circ s(g) \big)\nonumber\\
			&=& \big(\xi^{-1}_g (\xi_h (a)),~\epsilon_g (s(h)\circ b\circ s(h)^\dagger) \big), \text{by~\eqref{action1 sb} and~\eqref{action3 sb}}\nonumber\\
			&=& \big(\xi^{-1}_g (\xi_h (a)),~\epsilon_g (\epsilon^{-1}_h (b))\big), \text{using~\eqref{action3 sb}}\nonumber\\
			&=&\big(\xi_{(g^\dagger\circ h)}(a),~\epsilon_{(h^\dagger \circ g)}(b) \big),\label{tilde xi}
		\end{eqnarray}
		where the last equality follows from the fact that $\xi$ is a~homomorphism and $\epsilon$ is an anti-homomorphism.
		
		To compute $\widetilde{\zeta}$, we see that
		\begin{eqnarray}
			\widetilde{\zeta}_{(h,g)}(a,b)&=& (\widetilde{s}(h,g))^{-1} \bullet (a,b) \bullet \widetilde{s}(h,g )\nonumber\\
			&=& (s(h),s(g))^{-1} \bullet (a,b) \bullet (s(h),s(g) )\nonumber\\
			&=& \big(\lambda^E_{s(g)^\dagger}(s(h))^{-1},s(g)^\dagger \big) \bullet \big(a \cdot \lambda^E_b (s(h)),b\circ s(g) \big)\nonumber\\
			&=& \big(\lambda^E_{s(g)^\dagger}(s(h)^{-1} \cdot a~\cdot \lambda^E_b (s(h)) ),~s(g)^\dagger \circ b\circ s(g)\big)\nonumber\\
			&=& \big(\lambda^E_{s(g)^\dagger}(s(h)^{-1} \cdot a~\cdot s(h)\cdot s(h)^{-1}\cdot\lambda^E_b (s(h))),~\epsilon_g (b)\big),~\textrm{by}~\eqref{action3 sb}.\label{tilde zeta 1}
		\end{eqnarray}
		Note that,
		\begin{eqnarray*}
			s(h)^{-1} \cdot \lambda^E_b (s(h))&=& s(h)^{-1}\cdot \big(b^{-1}\cdot(b\circ s(h))\big)\\
			&=&s(h)^{-1}\cdot b^{-1}\cdot s(h)\cdot \lambda^E_{s(h)} \big(s(h)^\dagger\circ b\circ s(h)\big)\\
			&=&\zeta_h (b^{-1})\, \xi_h \big(\epsilon_h (b)\big), \text{using~\eqref{action1 sb},~\eqref{action2 sb} and~\eqref{action3 sb}}.
		\end{eqnarray*}
		Using the above equation in~\eqref{tilde zeta 1}, we get
		\begin{eqnarray}
			\widetilde{\zeta}_{(h,g)}(a,b) &=& \big(\lambda^E_{s(g)^\dagger}(\zeta_h (a)\zeta_h (b^{-1})\xi_h (\epsilon_h (b))),~\epsilon_g (b)\big), \text{using~\eqref{action2 sb}}\nonumber\\
			&=& \big(\xi^{-1}_g \big(\zeta_h (a)\zeta_h (b^{-1})\xi_h (\epsilon_h (b)) \big),~\epsilon_g (b)\big), \text{using~\eqref{action1 sb}.}\label{tilde zeta}
		\end{eqnarray}
		Similarly, to determine $\widetilde{\epsilon}$, we have
		\begin{eqnarray}
			\widetilde{\epsilon}_{(h,g)}(a,b)&=&(\widetilde{s}(h,g))^\dagger\odot (a,b)\odot \widetilde{s}(h,g)\nonumber\\
			&=&(s(h),s(g))^\dagger\odot(a,b)\odot(s(h),s(g))\nonumber\\
			&=&\big(s(h)^\dagger,s(h)^\dagger\circ s(g)^\dagger\circ s(h)\big)\odot \big(a\circ s(h),a\circ s(g)\circ a^\dagger \circ b\big)\nonumber\\
			&=&\big(s(h)^\dagger\circ a\circ s(h),~s(h)^\dagger\circ a\circ s(g)\circ a^\dagger\circ b\circ s(g)^\dagger \circ s(h)\big)\nonumber\\
			&=&\big(\epsilon_h (a),s(h)^\dagger\circ a\circ \epsilon^{-1}_g (a^\dagger \circ b) \circ s(h)\big), \text{using~\eqref{action3 sb}}\nonumber\\
			&=&\big(\epsilon_h (a),~\epsilon_h (a\circ \epsilon^{-1}_g (a^\dagger\circ b))\big).\label{tilde epsilon}
		\end{eqnarray}
		Hence, we conclude that the associated action $(\widetilde{\xi},\widetilde{\zeta},\widetilde{\epsilon})$ of the induced skew left brace is completely determined by the action $(\xi,\zeta,\epsilon)$ of the given skew left brace. A direct check shows that equivalent extensions map to equivalent extensions. By~\cite[Proposition 3.4]{MR4604853}, equivalent extensions induce identical associated actions, and the result follows.
	\end{proof}
	
	Let $(H,\cdot,\circ)$ be a~skew left brace and $I$ be an abelian group viewed as a~trivial brace. We show that there exists a~group homomorphism $\Ho^2_{SB}(H,I)\to\Ho^2_{SB}(\widetilde{H},\widetilde{I})$.
	
	Let $\mathcal{E}$ be an extension of $H$ by $I$ and $s$ be a~set-theoretic section to $\mathcal{E}$. Let $\widetilde{\mathcal{E}}$ be the induced extension of $\widetilde{H}$ by $\widetilde{I}$ and $\widetilde{s}$ given by $\widetilde{s}(h,g)=(s(h),s(g))$ be the induced set-theoretic section to $\widetilde{\mathcal{E}}$. Define $\tau^{(\tau_1,\tau_2 )}\colon \widetilde{H} \times \widetilde{H} \to \widetilde{I}$ by
	\[
	\tau^{(\tau_1,\tau_2)} \big((h_1,g_1),(h_2,g_2)\big)=\widetilde{s} \big((h_1,g_1) \bullet (h_2,g_2)\big)^{-1} \bullet \widetilde{s}(h_1,g_1) \bullet \widetilde{s}(h_2,g_2)
	\]
	for all $g_1, g_2, h_1, h_2 \in H$. Then, we have
	\begin{align}
		\tau^{(\tau_1,\tau_2 )}&\big((h_1,g_1 ),(h_2,g_2 )\big) \nonumber\\
		&= \big(\widetilde{s}(h_1\cdot \lambda^H_{g_1}(h_2 ),g_1\circ g_2 )\big)^{-1} \bullet \big(s(h_1 ),s(g_1 )\big) \bullet \big(s(h_2 ),s(g_2 )\big)\nonumber\\
		&= \big(s(h_1\cdot \lambda^H_{g_1}(h_2 )),s(g_1\circ g_2 )\big)^{-1} \bullet \big(s(h_1 )\cdot \lambda^E_{s(g_1 )}(s(h_2 )),s(g_1 )\circ s(g_2 ) \big)\nonumber\\
		&=\big(\lambda^E_{s(g_1\circ g_2 )^\dagger}(s(h_1\cdot \lambda^H_{g_1}(h_2 )))^{-1},s(g_1\circ g_2 )^\dagger\big) \bullet \big(s(h_1 )\cdot \lambda^E_{s(g_1 )}(s(h_2 )),s(g_1 )\circ s(g_2 )\big)\nonumber\\
		&=\big(\lambda^E_{s(g_1\circ g_2 )^\dagger} \big(s(h_1 \cdot \lambda^H_{g_1}(h_2 ))^{-1}\cdot s(h_1 )\cdot \lambda^E_{s(g_1 )}(s(h_2 ))\big),~s(g_1\circ g_2 )^\dagger\circ s(g_1 )\circ s(g_2 ) \big).\label{tau des}
	\end{align}
	We examine the terms independently. For the first component, we see that
	\begin{align}
		s&\big(h_1\cdot \lambda^H_{g_1}(h_2 )\big)^{-1} \cdot s(h_1 )\cdot \lambda^E_{s(g_1 )}(s(h_2 )) \nonumber\\
		&=\tau_1 (h_1,\lambda^H_{g_1}(h_2 )) \cdot s(\lambda^H_{g_1}(h_2 ))^{-1} \cdot \lambda^E_{s(g_1 )}(s(h_1 )), \text{using~\eqref{tau sb}}\nonumber\\
		&=\tau_1 (h_1,\lambda^H_{g_1}(h_2 ))\cdot (s(g_1^{-1}\cdot (g_1\circ h_2 )))^{-1}\cdot s(g_1 )^{-1} \cdot (s(g_1 )\circ s(h_2 ))\nonumber\\
		&=\tau_1 (h_1,\lambda^H_{g_1}(h_2 ))\cdot \tau_1 (g_1^{-1},g_1\circ h_2 ) \cdot s(g_1\circ h_2 )^{-1} \cdot s(g_1^{-1})^{-1}\cdot s(g_1 )^{-1} \cdot (s(g_1 )\circ s(h_2 )), \nonumber\\
		&\quad \text{ using \eqref{tau sb}}\nonumber\\
		&=\tau_1 (h_1,\lambda^H_{g_1}(h_2 ))\cdot \tau_1 (g_1^{-1},g_1\circ h_2 )\cdot s(g_1\circ h_2 )^{-1}\cdot \tau_1 (g_1,g_1^{-1})^{-1} \cdot (s(g_1 )\circ s(h_2 ))\nonumber\\
		&=\tau_1 (h_1,\lambda^H_{g_1}(h_2 ))\cdot \tau_1 (g_1^{-1},g_1\circ h_2 )\cdot s(g_1\circ h_2 )^{-1} \nonumber\\
		&\quad\cdot \tau_1 (g_1,g_1^{-1})^{-1}\cdot s(g_1\circ h_2 )\cdot s(g_1\circ h_2 )^{-1}\cdot (s(g_1 )\circ s(h_2 ))\nonumber\\
		&=\tau_1 (h_1,\lambda^H_{g_1}(h_2 ))\cdot \tau_1 (g_1^{-1},g_1\circ h_2 )\cdot \zeta_{(g_1\circ h_2 )}\big(\tau_1 (g_1,g_1^{-1})^{-1}\big)\cdot \tau_2 (g_1, h_2 ),\label{tau des1}
	\end{align}
	using~\eqref{tau sb},~\eqref{tildetau sb} and~\eqref{action2 sb}. 
\par

	For the second component, we have
	\begin{eqnarray}
		&& s(g_1\circ g_2 )^\dagger\circ \big(s(g_1 )\circ s(g_2 ) \big)\nonumber\\
		&=&s(g_1\circ g_2 )^\dagger \cdot \lambda^E_{s(g_1\circ g_2 )^\dagger} \big(s(g_1 )\circ s(g_2 ) \big)\nonumber\\
		&=&\lambda_{s(g_1\circ g_2 )^\dagger}(s(g_1\circ g_2 )^{-1}) \cdot \lambda^E_{s(g_1\circ g_2 )^\dagger} \big(s(g_1 )\circ s(g_2 )\big), \text{since $a^\dagger=\lambda^{-1}_a (a^{-1})$,}\nonumber\\
		&=&\lambda_{s(g_1\circ g_2 )^\dagger}\big(s(g_1\circ g_2 )^{-1} \cdot (s(g_1 )\circ s(g_2 ))\big)\nonumber\\
		&=&\xi^{-1}_{(g_1\circ g_2 )}(\tau_2 (g_1,g_2 )).\label{tau des2}
	\end{eqnarray}
	Using~\eqref{tau des1} and~\eqref{tau des2} in~\eqref{tau des}, we get
	\begin{align}
		\tau^{(\tau_1,\tau_2 )}&\big((h_1,g_1 ), \, (h_2,g_2 )\big) \nonumber\\
		&= \left( \xi^{-1}_{(g_1\circ g_2 )}(\tau_1 (h_1,\lambda^H_{g_1}(h_2 )) \cdot \tau_1 (g_1^{-1},g_1\circ h_2 ) \cdot\zeta_{(g_1\circ h_2 )}(\tau_1 (g_1,g_1^{-1})^{-1}) \cdot \tau_2 (g_1, h_2 )), \right. \nonumber \\
		&\qquad \left. \xi^{-1}_{(g_1\circ g_2 )}(\tau_2 (g_1,g_2 )) \right).
		\label{tau 2 prime}
	\end{align}
	Similarly, we define $\tau'^{(\tau_1,\tau_2 )}\colon \widetilde{H} \times \widetilde{H}\to \widetilde{I}$ by
	\[
	\tau'^{(\tau_1,\tau_2)} \big((h_1,g_1),(h_2,g_2)\big)=\big(\widetilde{s}((h_1,g_1)\odot(h_2,g_2))\big)^{-1} \bullet \big(\widetilde{s}(h_1,g_1)\odot \widetilde{s}(h_2,g_2)\big).
	\]
	Then, we have
	\begin{align}
		\tau'^{(\tau_1,\tau_2 )} & \big((h_1,g_1 ),(h_2,g_2 )\big) \nonumber\\
		={ }&{ }\big(\widetilde{s}(h_1\circ h_2, h_1\circ g_2\circ h_1^\dagger\circ g_1 )\big)^{-1} \bullet \big(s(h_1 )\circ s(h_2 ), s(h_1 )\circ s(g_2 )\circ s(h_1 )^\dagger\circ s(g_1 )\big)\nonumber\\
		={ }&{ } \big(s(h_1\circ h_2 ),~s(h_1\circ g_2\circ h_1^\dagger\circ g_1 )\big)^{-1} \bullet \big(s(h_1 )\circ s(h_2 ), s(h_1 )\circ s(g_2 )\circ s(h_1 )^\dagger\circ s(g_1 )\big)\nonumber\\
		={ }&{ }\big(\lambda^E_{s(h_1\circ g_2\circ h_1^\dagger\circ g_1 )^\dagger}(s(h_1\circ h_2 )^{-1}), \, s(h_1\circ g_2\circ h_1^\dagger\circ g_1 )^\dagger\big) \nonumber \\
		&\quad \big(s(h_1 )\circ s(h_2 ), \, s(h_1 )\circ s(g_2 )\circ s(h_1 )^\dagger\circ s(g_1 )\big)\nonumber\\
		={ }&{ }\big(\lambda^E_{s(h_1\circ g_2\circ h_1^\dagger\circ g_1 )^\dagger} (s(h_1\circ h_2 )^{-1}\cdot (s(h_1 )\circ s(h_2 ))), \nonumber\\
		&\quad s(h_1\circ g_2\circ h_1^\dagger\circ g_1 )^\dagger\circ s(h_1 )\circ s(g_2 )\circ s(h_1 )^\dagger\circ s(g_1 ) \big)\nonumber\\
		={ }&{ }\left(\xi^{-1}_{(h_1\circ g_2\circ h_1^\dagger\circ g_1 )}(\tau_2 (h_1,h_2 )), \, s(h_1\circ g_2\circ h_1^\dagger\circ g_1 )^\dagger\circ s(h_1 )\circ s(g_2 )\circ s(h_1 )^\dagger\circ s(g_1 )\right), \label{tilde tau}\nonumber\\
		&\quad\text{using \eqref{tildetau sb} and \eqref{action1 sb}.}
	\end{align}
	Computing the second component, we get
	\begin{eqnarray}
		&& s(h_1\circ g_2\circ h_1^\dagger\circ g_1 )^\dagger\circ s(h_1 )\circ s(g_2 )\circ s(h_1 )^\dagger\circ s(g_1 ) \nonumber\\
		&=&\big(s((h_1\circ g_2 )\circ (h_1^\dagger\circ g_1 ))^\dagger\circ s(h_1\circ g_2 )\circ s(h_1^\dagger\circ g_1 )\big)\circ\nonumber\\
		&& s(h_1^\dagger\circ g_1 )^\dagger\circ \big(s(h_1\circ g_2 )^\dagger\circ s(h_1 )\circ s(g_2 )\big)\circ s(h_1 )^\dagger\circ s(g_1 )\nonumber\\
		&=&\xi^{-1}_{h_1\circ g_2\circ h_1^\dagger\circ g_1}(\tau_2 \big(h_1\circ g_2,h_1^\dagger \circ g_1 ) \big)\circ s(h_1^\dagger\circ g_1 )^\dagger\circ \xi^{-1}_{(h_1\circ g_2 )}\big(\tau_2 (h_1,g_2 )\big)\circ s(h_1 )^\dagger\circ s(g_1 ), \nonumber\\
		&&\text{using \eqref{tau des2}}\nonumber\\
		&=&\xi^{-1}_{h_1\circ g_2\circ h_1^\dagger\circ g_1} \big(\tau_2 (h_1\circ g_2,h_1^\dagger \circ g_1 )\big)\circ s(h_1^\dagger\circ g_1 )^\dagger\circ\xi^{-1}_{(h_1\circ g_2 )}\big(\tau_2 (h_1,g_2 )\big)\circ s(h_1^\dagger\circ g_1 ) \circ \nonumber\\
		&&s(h_1^\dagger\circ g_1 )^\dagger\circ s(h_1 )^\dagger\circ s(g_1 )\nonumber\\
		&=&\xi^{-1}_{h_1\circ g_2\circ h_1^\dagger\circ g_1} \big(\tau_2 (h_1\circ g_2,h_1^\dagger \circ g_1 )\big)\circ \epsilon_{(h_1^\dagger\circ g_1 )} \big(\xi^{-1}_{(h_1\circ g_2 )}(\tau_2 (h_1,g_2 )) \big)\circ\nonumber\\
		&&s(h_1^\dagger\circ g_1 )^\dagger\circ s(h_1^\dagger)\circ \tau_2 (h_1,h_1^\dagger)^\dagger\circ s(g_1 ), \text{using~\eqref{action3 sb}}\nonumber \\
		&=&\xi^{-1}_{h_1\circ g_2\circ h_1^\dagger\circ g_1} \big(\tau_2 (h_1\circ g_2,h_1^\dagger \circ g_1 )\big)\circ \epsilon_{(h_1^\dagger\circ g_1 )}\big(\xi^{-1}_{(h_1\circ g_2 )}(\tau_2 (h_1,g_2 ))\big)\circ\nonumber\\
		&&s(h_1^\dagger\circ g_1 )^\dagger\circ s(h_1^\dagger)\circ s(g_1 )\circ s(g_1 )^\dagger\circ \tau_2 (h_1,h_1^\dagger)^\dagger\circ s(g_1 )\nonumber\\
		&=&\xi^{-1}_{h_1\circ g_2\circ h_1^\dagger\circ g_1}\big(\tau_2 (h_1\circ g_2,h_1^\dagger \circ g_1 )\big)\circ \epsilon_{(h_1^\dagger\circ g_1 )}\big(\xi^{-1}_{(h_1\circ g_2 )}(\tau_2 (h_1,g_2 ))\big)\circ\nonumber\\
		&&\xi^{-1}_{(h_1^\dagger\circ g_1 )}\big(\tau_2 (h_1^\dagger,g_1 )\big)\circ\epsilon_{g_1} \big(\tau_2 (h_1,h_1^\dagger)^\dagger\big),~\text{using~\eqref{action3 sb} and~\eqref{tau des2}.} \label{tilde tau1}
	\end{eqnarray}
	Using~\eqref{tilde tau1} in~\eqref{tilde tau}, we get
	\begin{eqnarray}
		&& \tau'^{(\tau_1,\tau_2 )}\big((h_1,g_1 ),(h_2,g_2 )\big) \nonumber\\
		&=&\big(\xi^{-1}_{(h_1\circ g_2\circ h_1^\dagger\circ g_1 )}(\tau_2 (h_1,h_2 )),~\xi^{-1}_{h_1\circ g_2\circ h_1^\dagger\circ g_1}(\tau_2 (h_1\circ g_2,h_1^\dagger \circ g_1 ))\circ \nonumber\\
		&& \epsilon_{(h_1^\dagger\circ g_1 )}(\xi^{-1}_{(h_1\circ g_2 )}(\tau_2 (h_1,g_2 ))) \circ \xi^{-1}_{(h_1^\dagger\circ g_1 )}(\tau_2 (h_1^\dagger,g_1 )) \circ \epsilon_{g_1}(\tau_2 (h_1,h_1^\dagger)^\dagger)\big)\nonumber\\
		&=&\big(\xi^{-1}_{(h_1\circ g_2\circ h_1^\dagger\circ g_1 )}(\tau_2 (h_1,h_2 )),~\xi^{-1}_{h_1\circ g_2\circ h_1^\dagger\circ g_1} (\tau_2 (h_1\circ g_2,h_1^\dagger \circ g_1 ))\circ\label{tau prime square}\\
		&& \epsilon_{(h_1^\dagger\circ g_1 )}(\xi^{-1}_{(h_1\circ g_2 )}(\tau_2 (h_1,g_2 )))\circ \xi^{-1}_{g_1}(\tau_2 (h_1,h_1^\dagger\circ g_1 ))^{-1}\big),\nonumber \\
		&& \text{using~\eqref{sbcocycle2} and choosing $m_1 =h_1, m_2 = h_1^\dagger \textrm{and}~m_3 = g_1$}.\nonumber
	\end{eqnarray}
	Let $\Upsilon^{-1}: \Ho^2_{SB}(H, I)\rightarrow \Ext_{(\xi, \zeta, \epsilon)}(H, I)$ and $\widetilde{\Upsilon}:\Ext_{(\widetilde{\xi}, \widetilde{\zeta}, \widetilde{\epsilon})}(\widetilde{H}, \widetilde{I}) \rightarrow \Ho^2_{SB}(\widetilde{H}, \widetilde{I})$ be the bijections given by Theorem~\ref{gbij-thm sb}. Let $\Theta: \Ext_{(\xi,\zeta,\epsilon)}(H, I) \to \Ext_{(\widetilde{\xi},\widetilde{\zeta},\widetilde{\epsilon})}(\widetilde{H},\widetilde{I})$ be the map defined in Corollary~\ref{extsb-extdoublesb}. With the preceding set-up, we have the following result.
	
	\begin{proposition}\label{homology homo SLB}
		Let $(H,\cdot,\circ)$ be a~skew left brace and $(I,\cdot)$ be a~module over $(H,\cdot,\circ)$ with respect to the action $(\xi,\zeta,\epsilon)$. Then the map
		\[
		\Omega_{SB}:=\widetilde{\Upsilon}\Theta\Upsilon^{-1}: \Ho^2_{SB}(H, I)\to\Ho^2_{SB}(\widetilde{H}, \widetilde{I})
		\]
		is a~homomorphism of groups.
	\end{proposition}
	
	\begin{proof}
		The map $\Omega_{SB}$ is explicitly given by $\Omega_{SB}([\tau_1,\tau_2 ])=[\tau^{(\tau_1,\tau_2 )},\tau'^{(\tau_1,\tau_2 )}]$, where $\tau^{(\tau_1,\tau_2 )}$ and $\tau'^{(\tau_1,\tau_2 )}$ are as described in~\eqref{tau 2 prime} and~\eqref{tau prime square}, respectively.	Let $[\tau_1, \tau_2 ]$ and $[\mu_1,\mu_2 ]$ be elements in $\Ho^2_{SB}(H,I)$. Then
		\[
		\Omega_{SB} \big( [\tau_1,\tau_2] \, [\mu_1,\mu_2] \big) = \Omega_{SB} \big( [\tau_1 \mu_1, \tau_2 \mu_2]\big) =\big[\tau^{(\tau_1 \mu_1, \tau_2 \mu_2)},\tau'^{(\tau_1 \mu_1, \tau_2\mu_2)} \big].
		\]
		It is easy to see that
		\[
		\tau^{(\tau_1 \mu_1, \tau_2 \mu_2)}=	\tau^{(\tau_1, \tau_2)}\, \tau^{(\mu_1,\mu_2)} \quad \textrm{and} \quad  \tau'^{(\tau_1 \mu_1,\tau_2 \mu_2)}=	\tau'^{(\tau_1, \tau_2)}\, \tau'^{(\mu_1, \mu_2)}.
		\]
		Hence, we obtain
		\[
		\Omega_{SB} \big( [\tau_1, \tau_2] \, [\mu_1,\mu_2] \big) = \big[\tau^{(\tau_1, \tau_2)},\tau'^{(\tau_1, \tau_2)}\big] \, \big[\tau^{(\mu_1, \mu_2)},\tau'^{(\mu_1,\mu_2)}\big]\\
		= \Omega_{SB} \big( [\tau_1, \tau_2]\big) \, \Omega_{SB} \big([\mu_1,\mu_2] \big),
		\]
		which shows that $\Omega_{SB}$ is a~homomorphism of groups.
	\end{proof}

	\section{Commutative diagram of second cohomology groups}\label{section comm diag homol groups}
	
	Let $(H,\cdot,\circ)$ be a skew left brace and $I$ an abelian group viewed as a trivial brace. Let $\mathcal{H}=(H^{(\cdot)},H^{(\circ)},\lambda^H, \Id_H)$ and  $\mathcal{I}=(I^{(\cdot)},I^{(\cdot)},\lambda^I, \Id_I)$, viewed as an $\mathcal{H}-$module with respect to the action $(\nu,\mu,\sigma,f)$. By \cite[Proposition 4.4 and Corollary 4.6]{MR4819014}, there is a group isomorphism $\Psi: \Ho^2_{RRB}(\mathcal{H},\mathcal{I})\to \Ho^2_{SB}(H,I)$ defined by
	\begin{eqnarray}\label{defn of Psi}
		\Psi \big([\tau_1,\tau_2,\rho,\chi] \big)&=&[\tau_1, \tau^{(\lambda^H, \Id_H)}_1\rho^{\Id_H}\chi^{(\Id_H, f)} ],
	\end{eqnarray}
	where
	\begin{eqnarray}
		\tau^{(\lambda^H, \Id_H)}_1(h_1, h_2) &=&  \tau_1 \big(h_1, \lambda^H_{h_1}(h_2)\big), \nonumber\\
		\rho^{\Id_H}(h_1, h_2) &=&\rho(h_2, h_1)~\textrm{and} \nonumber\\
		\chi^{(\Id_H, f)}(h_1, h_2) & =& \nu_{h_1} \big( f(\chi(h_1), h_2) \big) \label{defn of chi id f}
	\end{eqnarray}
	for all $h_1, h_2 \in H$. Here, the cohomology $\Ho^2_{SB}(H,I)$ is with respect to the induced action $(\xi,\zeta,\epsilon)$, where $\xi=\nu$, $\zeta=\mu$ and $\epsilon=\sigma$.
	
	Similarly, by \cite[Proposition 4.8]{MR4644858}, for the Rota--Baxter group  $\widetilde{H}$   and $\widetilde{I}$ viewed as an $\widetilde{H}-$module with associated action $\gamma$,  there is a group homomorphism
	\[
	\widetilde{\Psi}:\Ho^2_{RB}(\widetilde{H},\widetilde{I})\to\Ho^2_{SB}(\widetilde{H},\widetilde{I})
	\quad \text{defined by} \quad
	\widetilde{\Psi}\big([\tau,r]\big)=[\tau,\beta],
	\]
	where
	\begin{align}
		\beta \big((h_1,g_1),(h_2,g_2) \big)
		={ }&{ }\tau\big((h_1,g_1),i_{\widetilde{R}(h_1,g_1)} (h_2,g_2)\big)~\tau \big(\widetilde{R}(h_1,g_1) \bullet (h_2,g_2),
		(\widetilde{R}(h_1,g_1))^{-1}\big)~ \nonumber\\
		&\gamma_{(\widetilde{R}(h_1,g_1))^{-1}} \big(\tau \big(\widetilde{R}(h_1,g_1),(h_2,g_2) \big)\, \gamma_{(h_2,g_2)} \big(r(h_1,g_1)\big)~\big(r(h_1,g_1)\big)^{-1} \big)~\nonumber\\
		&\big(\tau \big(\widetilde{R}(h_1,g_1),(\widetilde{R}(h_1,g_1))^{-1} \big)\big)^{-1}\label{beta definition}
	\end{align}
	and $i_x$ is the inner automorphism given by $i_x(y) = xy x^{-1}$. Here, the cohomology $\Ho^2_{SB}(\widetilde{H},\widetilde{I})$ is with respect to the induced associated action $(\widetilde{\xi},\widetilde{\zeta},\widetilde{\epsilon})$. With the preceding set-up, we prove the following result.
	
	\begin{theorem}\label{theorem comm diag homol groups}
		Let $(H,\cdot,\circ)$ be a skew left brace and $I$ an abelian group viewed as a trivial brace. Let $\mathcal{H}=(H^{(\cdot)},H^{(\circ)},\lambda^H, \Id_H)$ and  $\mathcal{I}=(I^{(\cdot)},I^{(\cdot)},\lambda^I, \Id_I)$ viewed as an $\mathcal{H}-$module with respect to the action $(\nu,\mu,\sigma,f)$. Then the diagram
		\begin{align}
			\begin{CD}\label{commdiagram1}
				\Ho^2_{RRB}(\mathcal{H},\mathcal{I}) @>\Omega_{RB}>> \Ho^2_{RB}(\widetilde{H},\widetilde{I}) \\
				@V{\Psi}VV	@V{\widetilde{\Psi}}VV \\
				\Ho^2_{SB}(H,I) @>\Omega_{SB}>>\Ho^2_{SB}(\widetilde{H},\widetilde{I})
			\end{CD}
		\end{align}
		commutes, where $\Ho^2_{RB}$ and $\Ho^2_{SB}$ are the cohomologies with respect to the induced module structures.
	\end{theorem}
	
	\begin{proof}
		We first establish our notation for clarity. Let $\Ho^2_{RRB}(\mathcal{H},\mathcal{I})$ be the second cohomology of  $\mathcal{H}$ with coefficients in the module  $\mathcal{I}$ with respect to the action $(\nu,\mu,\sigma,f)$, where the map $f$ is trivial. 	 For $[\tau_1,\tau_2,\rho,\chi]\in \Ho^2_{RRB}(\mathcal{H},\mathcal{I})$, we write
		$$\Psi \big([\tau_1,\tau_2,\rho,\chi] \big)=[\tau_1,\hat{\tau_2}] \in \Ho^2_{SB}(H,I),$$
		where the cohomology $\Ho^2_{SB}(H,I)$ is with respect to the induced action $(\xi,\zeta,\epsilon)$. Setting $\varphi=(\tau_1,\hat{\tau_2})$, we write
		$$\Omega_{SB}\big([\tau_1,\hat{\tau_2}] \big)=[\tau^\varphi,\tau'^\varphi] \in\Ho^2_{SB}(\widetilde{H},\widetilde{I}),$$
		with the induced associated action $(\tilde{\xi},\tilde{\zeta},\tilde{\epsilon})$ as defined in \eqref{tilde xi}, \eqref{tilde zeta} and \eqref{tilde epsilon}.
		
		In view of Proposition \ref{cohom RRB to cohom RB}, for $[\tau_1,\tau_2,\rho,\chi]\in \Ho^2_{RRB}(\mathcal{H},\mathcal{I})$, setting $\kappa=(\tau_1,\tau_2,\rho,\chi)$, we write
		$$\Omega_{RB} \big([\tau_1,\tau_2,\rho,\chi]\big)= \big[\tau^{\kappa},r^{\kappa}\big] \in \Ho^2_{RB}(\widetilde{H},\widetilde{I}),$$
		where the cohomology $\Ho^2_{RB}(\widetilde{H},\widetilde{I})$  is with respect to the induced action $\gamma$ given by \eqref{gamma action rb}. Further, for $\big[\tau^{\kappa},r^{\kappa}\big]\in \Ho^2_{RB}(\widetilde{H},\widetilde{I})$, we write
		$$\widetilde{\Psi}\big([\tau^{\kappa},r^{\kappa}]\big)=[\tau^\kappa,\beta^{r^\kappa}] \in\Ho^2_{SB}(\widetilde{H},\widetilde{I}),$$
		where the cohomology $\Ho^2_{SB}(\widetilde{H},\widetilde{I})$ is with respect to the induced action  $(\widetilde{\xi},\widetilde{\zeta},\widetilde{\epsilon})$.  It is a direct observation that  the induced action defining the cohomology  $\Ho^2_{SB}(\widetilde{H},\widetilde{I})$ is identical from both the directions of the diagram \eqref{commdiagram1}.
		
		Let $\widetilde{R}$ be the induced Rota--Baxter operator on $\widetilde{H}$ as given in \eqref{RB operator on semi-direct product}. Using \eqref{beta definition}, we compute the terms of $\beta^{r^\kappa} \big((h_1,g_1),(h_2,g_2)\big)$ individually. First, we have
		\begin{small}
			\begin{eqnarray}
				&& 	\tau^{\kappa}\big((h_1,g_1),i_{\widetilde{R}(h_1,g_1)}(h_2,g_2)\big) \nonumber\\
				&=&\tau^{\kappa}\big((h_1,g_1),(1,g_1^\dagger\circ h_1) \bullet (h_2,g_2) \bullet (1,h_1^\dagger\circ g_1)\big)\nonumber\\
				&=&\tau^{\kappa}\big((h_1,g_1),(\lambda^H_{g_1^\dagger\circ h_1}(h_2),g_1^\dagger\circ h_1\circ g_2) \bullet (1,h_1^\dagger\circ g_1)\big)\nonumber\\
				&=&\tau^{\kappa}\big((h_1,g_1),(\lambda^H_{g_1^\dagger\circ h_1}(h_2),g_1^\dagger\circ h_1\circ g_2\circ h_1^\dagger\circ g_1)\big)\nonumber\\
				& =&\big(\nu^{-1}_{(h_1\circ g_2\circ h_1^\dagger\circ g_1)}\big(\tau_1(h_1,\lambda^H_{g_1}(\lambda^H_{g_1^\dagger\circ h_1}(h_2))) \,\rho(\lambda^H_{g_1^\dagger\circ h_1}(h_2),g_1)\big), \tau_2(g_1,g_1^\dagger\circ h_1\circ g_2\circ h_1^\dagger\circ g_1)\big), \nonumber\\
				&&\text{using  \eqref{tau 1 prime}}\nonumber\\
				& =&\big(\nu^{-1}_{(h_1\circ g_2\circ h_1^\dagger\circ g_1)}\big(\tau_1(h_1,\lambda^H_{ h_1}(h_2))\, \rho(\lambda^H_{g_1^\dagger\circ h_1}(h_2),g_1)\big), \tau_2(g_1,g_1^\dagger\circ h_1\circ g_2\circ h_1^\dagger\circ g_1)\big).\label{beta term 1}
			\end{eqnarray}
		\end{small}
		
		Second, we have
		\begin{eqnarray}
			&&	\tau^\kappa\big(\widetilde{R}(h_1,g_1) \bullet (h_2,g_2),(\widetilde{R}(h_1,g_1))^{-1}\big) \nonumber\\
			&=&\tau^\kappa\big((1,g_1^\dagger\circ h_1) \bullet (h_2,g_2),(1,h_1^\dagger\circ g_1)\big)\nonumber\\
			&=&\tau^\kappa\big((\lambda^H_{g_1^\dagger\circ h_1}(h_2),g_1^\dagger\circ h_1\circ g_2),(1,h_1^\dagger\circ g_1)\big)\nonumber\\
			&=& \big(\nu^{-1}_{(g_1^\dagger\circ h_1\circ g_2\circ h_1^\dagger\circ g_1)}\big(\tau_1(\lambda^H_{g_1^\dagger\circ h_1}(h_2),1) \, \rho(1,g_1^\dagger\circ h_1\circ g_2)\big), \tau_2(g_1^\dagger\circ h_1\circ g_2,h_1^\dagger\circ g_1) \big),\nonumber\\
			&&\text{using \eqref{tau 1 prime}}\nonumber\\
			&=& \big(1,\tau_2(g_1^\dagger\circ h_1\circ g_2,h_1^\dagger\circ g_1) \big),~\text{since ~$\tau_1(\lambda^H_{g_1^\dagger\circ h_1}(h_2),1)=1=\rho(1,g_1^\dagger\circ h_1\circ g_2)$}.\label{beta term 2}
		\end{eqnarray}
		Third, we have
		\begin{eqnarray}
			&& \tau^\kappa\big(\widetilde{R}(h_1,g_1),(h_2,g_2)\big) \nonumber\\
			&=&\tau^\kappa\big((1,g_1^\dagger\circ h_1),(h_2,g_2)\big)\nonumber\\
			&=&\big(\nu^{-1}_{g_1^\dagger\circ h_1\circ g_2}\big(\tau_1(1,\lambda^H_{g_1^\dagger\circ h_1}(h_2)) \, \rho(h_2,g_1^\dagger\circ h_1)\big),\tau_2(g_1^\dagger\circ h_1,g_2)\big),\text{using \eqref{tau 1 prime}}\nonumber\\
			&=&\big(\nu^{-1}_{g_1^\dagger\circ h_1\circ g_2}\big(\rho(h_2,g_1^\dagger\circ h_1)\big),\tau_2(g_1^\dagger\circ h_1,g_2)\big), ~\text{since ~$\tau_1(1,\lambda^H_{g_1^\dagger\circ h_1}(h_2))=1$.}\label{beta term 3.1}
		\end{eqnarray}
		Using \eqref{beta term 3.1}, we get
		\begin{eqnarray}
			&& \gamma_{(\widetilde{R}(h_1,g_1))^{-1}} \big(\tau^\kappa \big(\widetilde{R}(h_1,g_1),(h_2,g_2) \big)\big)\nonumber\\
			&=&\gamma_{(1,h_1^\dagger\circ g_1)}\big(\nu^{-1}_{g_1^\dagger\circ h_1\circ g_2}\big(\rho(h_2,g_1^\dagger\circ h_1)\big),\tau_2(g_1^\dagger\circ h_1,g_2)\big)\nonumber\\
			&=& \big(\nu^{-1}_{h_1^\dagger\circ g_1}\big(\mu_1(\nu^{-1}_{g_1^\dagger\circ h_1\circ g_2}\big(\rho(h_2,g_1^\dagger\circ h_1)) \, f(\tau_2(g_1^\dagger\circ h_1,g_2),1) \big),\, \sigma_{h_1^\dagger\circ g_1}(\tau_2(g_1^\dagger\circ h_1,g_2))\big), \nonumber\\
			&&\text{ using \eqref{gamma action rb}}\nonumber\\
			\quad \quad &=&\big(\nu^{-1}_{h_1^\dagger\circ g_1}(\nu^{-1}_{g_1^\dagger\circ h_1\circ g_2}\big(\rho(h_2,g_1^\dagger\circ h_1))),\,\sigma_{h_1^\dagger\circ g_1}(\tau_2(g_1^\dagger\circ h_1,g_2))\big),\label{beta term 3}\\
			&&  ~\textrm{since $\mu_1=\Id_I$ and $f(\tau_2(g_1^\dagger\circ h_1,g_2),1)=1$.}\nonumber
		\end{eqnarray}
		Considering the remaining three terms, we have
		\begin{eqnarray}
			&& \gamma_{(\widetilde{R}(h_1,g_1))^{-1}} \big( \gamma_{(h_2,g_2)}r^\kappa(h_1,g_1)\big)\nonumber\\
			&=&\gamma_{(1,h_1^\dagger\circ g_1)} \big(\gamma_{(h_2,g_2)}\big(1,\tau_2(g_1,g_1^\dagger\circ h_1)^{-1}\chi(h_1)\big)\big),\text{ using \eqref{g cocycle term}}\nonumber\\
			&=& \gamma_{(1,h_1^\dagger\circ g_1)} \big(\nu^{-1}_{g_2} \big(\mu_{h_2}(1)f(\tau_2(g_1,g_1^\dagger \circ h_1)^{-1}\chi(h_1),h_2)\big),\, \sigma_{g_2}(\tau_2(g_1,g_1^\dagger\circ h_1)^{-1}\chi(h_1))\big),\nonumber\\
			&&\text{ using \eqref{gamma action rb}}\nonumber\\
			&=& \gamma_{(1,h_1^\dagger\circ g_1)} \big(\nu^{-1}_{g_2} \big(f(\tau_2(g_1,g_1^\dagger\circ h_1)^{-1}\chi(h_1),h_2)\big), \, \sigma_{g_2}(\tau_2(g_1,g_1^\dagger\circ h_1)^{-1}\chi(h_1))\big)\nonumber\\
			\quad &=&\big(\nu^{-1}_{g_2\circ h_1^{\dagger}\circ g_1}(f(\tau_2(g_1,g_1^\dagger\circ h_1)^{-1}\chi(h_1),h_2)),\sigma_{g_2\circ h_1^{\dagger}\circ g_1}(\tau_2(g_1,g_1^\dagger\circ h_1)^{-1}\chi(h_1))\big)\label{beta term 4},\\
			&&\text{ using \eqref{gamma action rb}.}\nonumber
		\end{eqnarray}
		Using \eqref{gamma action rb} and \eqref{g cocycle term}, we have
		\begin{eqnarray}
			\gamma_{(\widetilde{R}(h_1,g_1))^{-1}} \big((r^\kappa(h_1,g_1))^{-1}\big)
			= \big(1,\sigma_{h_1^{\dagger}\circ g_1}(\chi(h_1)^{-1}\tau_2(g_1,g_1^\dagger\circ h_1)) \big) \label{beta term 5},
		\end{eqnarray}
		and
		\begin{eqnarray}
			&&\tau^\kappa\big(\widetilde{R}(h_1,g_1),(\widetilde{R}(h_1,g_1))^{-1}\big)^{-1}\nonumber\\
			&=&\tau^\kappa\big((1,g_1^\dagger\circ h_1),(1,h_1^\dagger\circ g_1)\big)^{-1}\nonumber\\
			&=&\big(1,\tau_2(g_1^\dagger\circ h_1,h_1^\dagger\circ g_1) \big)^{-1},\text{ using \eqref{tau 1 prime}}\nonumber\\
			&=& \big(1,\tau_2(g_1^\dagger\circ h_1,h_1^\dagger\circ g_1)^{-1} \big).\label{beta term 6}
		\end{eqnarray}
		
		Using \eqref{beta term 1}, \eqref{beta term 2}, \eqref{beta term 3}, \eqref{beta term 4}, \eqref{beta term 5} and  \eqref{beta term 6}, we can write
		\begin{equation*}
			\beta^{r^\kappa} \big((h_1,g_1),(h_2,g_2) \big)=\big(\beta^{r^\kappa}_1 \big((h_1,g_1),(h_2,g_2) \big),\, \beta^{r^\kappa}_2 \big((h_1,g_1),(h_2,g_2) \big)\big),
		\end{equation*}
		where
		\begin{eqnarray}
			\beta^{r^\kappa}_1 \big((h_1,g_1),(h_2,g_2) \big)&=&\nu_{(g_1^\dagger\circ h_1\circ g_2^\dagger\circ h_1^\dagger)}\big(\tau_1(h_1,\lambda^H_{ h_1}(h_2))\,\rho(\lambda^H_{g_1^\dagger\circ h_1}(h_2),g_1) \big)\nonumber\\
			&&\nu_{(g_1^\dagger\circ h_1\circ g_2^\dagger\circ h_1^\dagger )} \big(\nu_{g_1}(\rho(h_2,g_1^\dagger\circ h_1)) \big)\nonumber\\
			&&\nu^{-1}_{g_2\circ h_1^{\dagger}\circ g_1} \big(f(\tau_2(g_1,g_1^\dagger\circ h_1)^{-1}\chi(h_1),h_2)\big)\label{beta 1 exp}
		\end{eqnarray}
		and
		\begin{eqnarray*}
			\beta^{r^\kappa}_2 \big((h_1,g_1),(h_2,g_2) \big)&=&\tau_2(g_1,g_1^\dagger\circ h_1\circ g_2\circ h_1^\dagger\circ g_1) \, \tau_2(g_1^\dagger\circ h_1\circ g_2,h_1^\dagger\circ g_1)\nonumber\\
			&&\sigma_{h_1^\dagger\circ g_1} \big(\tau_2(g_1^\dagger\circ h_1,g_2) \big) \, \sigma_{g_2\circ h_1^{\dagger}\circ g_1}\big(\tau_2(g_1,g_1^\dagger\circ h_1)^{-1}\chi(h_1)\big)\nonumber\\
			&&\sigma_{h_1^{\dagger}\circ g_1}\big(\chi(h_1)^{-1}\tau_2(g_1,g_1^\dagger\circ h_1)\big)\, \tau_2(g_1^\dagger\circ h_1,h_1^\dagger\circ g_1)^{-1}.\label{beta 2 exp}
		\end{eqnarray*}
		We now simplify these expressions. Using appropriate substitutions, we get
		\begin{eqnarray*}
			&& \nu_{(g_1^\dagger\circ h_1\circ g_2^\dagger\circ h_1^\dagger)} \big(\rho(\lambda^H_{g_1^\dagger\circ h_1}(h_2),g_1)\, \nu_{g_1}(\rho(h_2,g_1^\dagger\circ h_1)) \big)\\
			&=&\nu_{(g_1^\dagger\circ h_1\circ g_2^\dagger\circ h_1^\dagger)} \big(\rho(h_2,h_1)\, \nu_{h_1}(f(\tau_2(g_1,g_1^\dagger\circ h_1),h_2))\big),~\text{using \eqref{cocycle3}}\\
			&=&\nu_{(g_1^\dagger\circ h_1\circ g_2^\dagger\circ h_1^\dagger)} \big(\rho(h_2,h_1)\big)\nu^{-1}_{g_2\circ h_1^\dagger\circ g_1}\big( (f(\tau_2(g_1,g_1^\dagger\circ h_1),h_2))\big).\\
		\end{eqnarray*}
		
		Using the preceding equality in \eqref{beta 1 exp}, we get
		\begin{eqnarray}\label{final beta 1 exp}
			&& \beta^{r^\kappa}_1 \big((h_1,g_1),(h_2,g_2) \big)\nonumber\\
			&=& \nu_{(g_1^\dagger\circ h_1\circ g_2^\dagger\circ h_1^\dagger)} \big(\tau_1(h_1,\lambda^H_{ h_1}(h_2))\rho(h_2,h_1)\big) \, \nu_{g_1^{\dagger}\circ h_1 \circ g_2^{\dagger}} \big(f(\chi(h_1),h_2)\big),\label{beta 1 final exp}\\
			&&\text{using Lemma \eqref{properties of f}(2).}\nonumber
		\end{eqnarray}
		For the expression $\beta^{r^\kappa}_2$, we have
		\begin{eqnarray}
			&&\beta^{r^\kappa}_2 \big((h_1,g_1),(h_2,g_2) \big)\nonumber\\
			&=&\tau_2(g_1,g_1^\dagger\circ h_1\circ g_2\circ h_1^\dagger\circ g_1) \, \tau_2(g_1^\dagger\circ h_1\circ g_2,h_1^\dagger\circ g_1)\nonumber\\
			&&\sigma_{h_1^\dagger\circ g_1} \big(\tau_2(g_1^\dagger\circ h_1,g_2)\big) \, \sigma_{g_2\circ h_1^{\dagger}\circ g_1}\big(\tau_2(g_1,g_1^\dagger\circ h_1)^{-1}\chi(h_1)\big)\nonumber\\
			&&\sigma_{h_1^{\dagger}\circ g_1}\big(\chi(h_1)^{-1}\tau_2(g_1,g_1^\dagger\circ h_1) \big) \, \tau_2(g_1^\dagger\circ h_1,h_1^\dagger\circ g_1)^{-1}\nonumber\\
			&=&\tau_2(h_1\circ g_2, h_1^{\dagger}\circ g_1)\, \sigma_{h_1^{\dagger}\circ g_1} \big(\tau_2(g_1,g_1^{\dagger}\circ h_1\circ g_2)\big)\nonumber\\
			&&\sigma_{h_1^\dagger\circ g_1} \big(\tau_2(g_1^\dagger\circ h_1,g_2)\big)\, \sigma_{g_2\circ h_1^{\dagger}\circ g_1}\big(\tau_2(g_1,g_1^\dagger\circ h_1)^{-1}\chi(h_1)\big)\nonumber\\
			&&\sigma_{h_1^{\dagger}\circ g_1}\big(\chi(h_1)^{-1}\, \tau_2(g_1,g_1^\dagger\circ h_1) \big)\, \tau_2(g_1^\dagger\circ h_1,h_1^\dagger\circ g_1)^{-1},\nonumber\\
			&&\text{using \eqref{cocycle2} with appropriate substitution}\nonumber\\
			&=&\tau_2(h_1\circ g_2, h_1^{\dagger}\circ g_1)\, \sigma_{h_1^{\dagger}\circ g_1} \big(\tau_2(h_1,g_2)\big)\nonumber\\
			&&\sigma_{h_1^{\dagger}\circ g_1} \big(\sigma_{g_2}(\chi(h_1))\chi(h_1)^{-1}\big)\, \tau_2(h_1,h_1^{\dagger}\circ g_1)^{-1},\label{beta 2 final exp}\\
			&&\text{using \eqref{cocycle2} twice with appropriate substitutions}.\nonumber
		\end{eqnarray}
		Note that, for $[\tau_1,\tau_2,\rho,\chi]\in \Ho^2_{RRB}(\mathcal{H},\mathcal{I})$, we have $[\tau^\kappa,\beta^{r^\kappa}]=\widetilde{\Psi} \, \Omega_{RB}\big ([\tau_1,\tau_2,\rho,\chi]\big)$ and $[\tau^\varphi,\tau'^\varphi]=\Omega_{SB} \, \Psi \big([\tau_1,\tau_2,\rho,\chi]\big)$. If $[\tau^\varphi,\tau'^\varphi]^{-1}$ denotes the inverse of $[\tau^\varphi,\tau'^\varphi]$ in $\Ho^2_{SB}(\widetilde{H},\widetilde{I})$, then we have
		\begin{equation}
			[\tau^\kappa,\beta^{r^\kappa}][\tau^\varphi,\tau'^\varphi]^{-1}=[\tau^\kappa( \tau^\varphi)^{-1},\beta^{r^\kappa} (\tau'^\varphi)^{-1}].
		\end{equation}
		We claim that $(\tau^\kappa( \tau^\varphi)^{-1},\, \beta^{r^\kappa} (\tau'^\varphi)^{-1}) \in \Z_{SB}^2(\widetilde{H},\widetilde{I})$.  First, we see that
		\begin{align*}
			\tau^\kappa & \big((h_1,g_1),(h_2,g_2)\big) \big(\tau^\varphi((h_1,g_1),(h_2,g_2))\big)^{-1}\\
			&=\big(\nu^{-1}_{g_1\circ g_2}\big(\tau_1(h_1,\lambda^H_{g_1}(h_2))\rho(h_2,g_1)\big),\tau_2(g_1,g_2)\big)\\
			&\quad\ \big(\xi^{-1}_{(g_1\circ g_2)} \big(\tau_1(h_1,\lambda^H_{g_1}(h_2))\tau_1(g_1^{-1},g_1\circ h_2)\zeta_{(g_1\circ h_2)}(\tau_1(g_1,g_1^{-1})^{-1})\hat{\tau_2}(g_1, h_2)\big),\\
			&\qquad\ \xi^{-1}_{(g_1\circ g_2)}(\hat{\tau_2}(g_1,g_2))\big)^{-1}, \text{using \eqref{tau 1 prime} and \eqref{tau 2 prime},}
		\end{align*}
		for all $h_1, h_2, g_1, g_2 \in H$. By \eqref{defn of Psi}, we have
		\begin{equation}\label{boundary eq 1}
			\hat{\tau_2}(g_1, h_2)=\tau_1(g_1,\lambda^H_{g_1}(h_2))\, \rho(h_2,g_1)\, \nu_{g_1}(f(\chi(g_1),h_2)).
		\end{equation}
		Using \eqref{cocycle1} with appropriate substitutions, we get
		\begin{equation}\label{boundary eq 2}
			\tau_1(g_1^{-1},g_1\circ h_2)\, \mu_{(g_1\circ h_2)}(\tau_1(g_1,g_1^{-1})^{-1}) \, \tau_1(g_1,\lambda^H_{g_1}(h_2))=1
		\end{equation}
		By \cite[Corollary 4.2 and Proposition 4.3]{MR4819014}, we have $\xi_h=\nu_h$ , $\zeta_h=\mu_h $ and  $\sigma_h=\epsilon_h$ for all $h\in H$. This gives
		\begin{align*}
			\big(&\nu^{-1}_{g_1\circ g_2}\big(\tau_1(h_1,\lambda^H_{g_1}(h_2))\rho(h_2,g_1)\big)\big)\\ &\big(\xi^{-1}_{(g_1\circ g_2)}(\tau_1(h_1,\lambda^H_{g_1}(h_2))\tau_1(g_1^{-1},g_1\circ h_2)	\zeta_{(g_1\circ h_2)}(\tau_1(g_1,g_1^{-1})^{-1})	\hat{\tau_2}(g_1, h_2)) \big)^{-1} \\
			&\!\!\!\!=\big(\nu^{-1}_{g_1\circ g_2}\big(\tau_1(h_1,\lambda^H_{g_1}(h_2))\rho(h_2,g_1)\big)\big) \\ &\quad\big(\nu^{-1}_{(g_1\circ g_2)}(\tau_1(h_1,\lambda^H_{g_1}(h_2))\tau_1(g_1^{-1},g_1\circ h_2) \mu_{(g_1\circ h_2)}(\tau_1(g_1,g_1^{-1})^{-1})	\hat{\tau_2}(g_1, h_2))	\big)^{-1}\\
			&\!\!\!\!=\big(\nu_{g_2^\dagger}(f(\chi(g_1),h_2))\big)^{-1}, \text{ using \eqref{boundary eq 1} and \eqref{boundary eq 2}.}
		\end{align*}
		Similarly, using \eqref{boundary eq 1} and \eqref{cocycle5}, we get
		\begin{eqnarray}
			\tau_2(g_1,g_2) \, \big(\xi^{-1}_{(g_1\circ g_2)}(	\hat{\tau_2}(g_1,g_2)) \big)^{-1}=\chi(g_2)^{-1} \, \chi(g_1\circ g_2) \, \sigma_{g_2}(\chi(g_1)^{-1}).\label{tau tau tilde rel}
		\end{eqnarray}
		Thus, we have
		\begin{eqnarray}
			&&\tau^\kappa \big((h_1,g_1),(h_2,g_2)\big) \, \big(\tau^\varphi((h_1,g_1),(h_2,g_2))\big)^{-1}\nonumber\\
			&=&\big( \big(\nu_{g_2^\dagger}(f(\chi(g_1),h_2))\big)^{-1},\,	\chi(g_2)^{-1}\chi(g_1\circ g_2)\sigma_{g_2}(\chi(g_1)^{-1})\big).\label{tau 1 prime tau 2 prime inverse 1}
		\end{eqnarray}
		
		Similarly, using \eqref{tau prime square}, \eqref{beta 1 final exp} and \eqref{beta 2 final exp}, we see that:
		\[
		\beta^{r^\kappa}\left( (h_1,g_1), \, (h_2,g_2) \right) \, \left(\tau'^\varphi((h_1,g_1), (h_2,g_2))\right)^{-1} = (h, g),
		\]
		where
		\begin{align*}
			h = { }&{ }\nu_{(g_1^\dagger\circ h_1\circ g_2^\dagger\circ h_1^\dagger)} \, (\tau_1(h_1,\lambda^H_{h_1}(h_2)) \rho(h_2,h_1)) \\&{ }\nu_{g_1^{\dagger} \circ h_1 \circ g_2^{\dagger}}(f(\chi(h_1),h_2)) \, \left(\xi^{-1}_{(h_1\circ g_2\circ h_1^\dagger\circ g_1)} (\hat{\tau_2}(h_1,h_2))\right)^{-1}
		\end{align*}
		and
		\begin{align*}
			g = { }&{ } \tau_2(h_1\circ g_2, h_1^{\dagger}\circ g_1) \, \sigma_{h_1^{\dagger}\circ g_1}(\tau_2(h_1,g_2)) \, \sigma_{h_1^{\dagger}\circ g_1}(\sigma_{g_2} (\chi(h_1)) \chi(h_1)^{-1}) \,  \tau_2(h_1,h_1^{\dagger}\circ g_1)^{-1} \\
			& \xi^{-1}_{h_1\circ g_2\circ h_1^\dagger\circ g_1} (\hat{\tau_2}(h_1\circ g_2,h_1^\dagger \circ g_1))\circ \epsilon_{(h_1^\dagger\circ g_1 )}(\xi^{-1}_{(h_1\circ g_2)}(\hat{\tau_2}(h_1,g_2)))\circ \xi^{-1}_{g_1}(\hat{\tau_2}(h_1,h_1^\dagger\circ g_1))^{-1}\big)^{-1}\!.
		\end{align*}
		Using \eqref{boundary eq 1},  and the fact that $\xi_h=\nu_h$ for all $h\in H$, the first component becomes trivial. Further, by multiple use of \eqref{tau tau tilde rel}, we obtain
		\begin{eqnarray}
			&&	\beta^{r^\kappa} \big((h_1,g_1),(h_2,g_2) \big)\,  \big(\tau'^\varphi((h_1,g_1),(h_2,g_2)) \big)^{-1} \nonumber\\
			&=&  \big(1,\, \chi(g_1)^{-1}\chi(h_1\circ g_2\circ h_1^{\dagger}\circ g_1) \, \sigma_{h_1^{\dagger}\circ g_1}(\chi(g_2))^{-1}  \big). \label{tau 1 prime tau 2 prime inverse 2}
		\end{eqnarray}
		
		Consider the map $\theta:\widetilde{H}\to I\times I$ given by
		$$\theta(h,g)= \big(1,\chi(g)^{-1} \big)$$
		for all $h, g \in H$. In view of \eqref{boundary con sb1} and \eqref{boundary con sb2}, to prove the claim, it suffices to show that
		\begin{align*}
			\tau^\kappa & \big((h_1,g_1),(h_2,g_2)\big) \, \big(\tau^\varphi((h_1,g_1),(h_2,g_2))\big)^{-1} \\
			&= \theta \big((h_1,g_1)\bullet(h_2,g_2)\big)^{-1} \widetilde{\zeta}_{(h_2,g_2)} \big(\theta(h_1,g_1)\big)\,\theta(h_2,g_2)
		\end{align*}
		and
		\begin{align*}
			\beta^{r^\kappa}&\big((h_1,g_1),(h_2,g_2) \big)\,  \big(\tau'^\varphi((h_1,g_1),(h_2,g_2)) \big)^{-1} \\
			={ }&{ }\theta \big((h_1,g_1)\odot(h_2,g_2)\big)^{-1} ~ \widetilde{\xi}_{(h_1,g_1)\odot(h_2,g_2)} \big(\widetilde{\epsilon}_{(h_2,g_2)}(\widetilde{\xi}^{-1}_{(h_1,g_1)}(\theta(h_1,g_1)))\big) \widetilde{\xi}_{(h_1,g_1)}\big(\theta(h_2,g_2)\big)
		\end{align*}
		for all $h_1,g_1,h_2,g_2 \in H$.  First, note that
		$$	\theta \big((h_1,g_1)\bullet(h_2,g_2) \big)^{-1}= \theta \big(h_1\lambda^H_{g_1}(h_2),\, g_1\circ g_2 \big)^{-1}= \big(1, \,\chi(g_1\circ g_2)^{-1}\big)^{-1} \nonumber\\
		= \big(1, \,\chi(g_1\circ g_2)\big)\nonumber
		$$
		and
		\begin{eqnarray*}
			\widetilde{\zeta}_{(h_2,g_2)} \big(\theta(h_1,g_1)\big)&=& \widetilde{\zeta}_{(h_2,g_2)} \big(1,\chi(g_1)^{-1}\big)\\
			&=&\big(\xi^{-1}_{g_2}(\zeta_{h_2}(\chi(g_1))\xi_{h_2}(\epsilon_{h_2}(\chi(g_1)^{-1}))), \, \epsilon_{g_2}(\chi(g_1)^{-1})\big), \text{ using \eqref{tilde zeta}.}
		\end{eqnarray*}
		By  \cite[Corollary 4.2 and Proposition 4.3]{MR4819014}, $\nu_h=\xi_h$ and $f(k,h)=\zeta_h(k^{-1})\xi_h(\epsilon_h(k))$, for all $h\in H$ and $k\in I$. Therefore,
		\begin{eqnarray*}
			\widetilde{\zeta}_{(h_2,g_2)} \big(\theta(h_1,g_1)\big)&=& \big(\nu_{g_2^\dagger}(f(\chi(g_1)^{-1},h_2)), \, \epsilon_{g_2}(\chi(g_1)^{-1}) \big)\\
			&=&\big(\big(\nu_{g_2^\dagger}(f(\chi(g_1),h_2))\big)^{-1}, \, \epsilon_{g_2}(\chi(g_1)^{-1})\big), ~\text{using Lemma \eqref{properties of f}(2).}
		\end{eqnarray*}
		
		Thus, we have
		\begin{eqnarray*}
			&& \theta \big((h_1,g_1)\bullet(h_2,g_2)\big)^{-1} ~  \widetilde{\zeta}_{(h_2,g_2)} \big(\theta(h_1,g_1)\big)\,\theta(h_2,g_2)\\
			&=& \big(1, \,\chi(g_1\circ g_2)\big)\big(\big(\nu_{g_2^\dagger}(f(\chi(g_1),h_2))\big)^{-1}, \, \epsilon_{g_2}(\chi(g_1)^{-1})\big) \big(1, \, \chi(g_2)^{-1}\big)\\
			&=& \big(\big(\nu_{g_2^\dagger}(f(\chi(g_1),h_2))\big)^{-1}, \,\chi(g_1\circ g_2) \epsilon_{g_2}(\chi(g_1)^{-1}) \chi(g_2)^{-1}\big)\\
			&=& \big(\big(\nu_{g_2^\dagger}(f(\chi(g_1),h_2))\big)^{-1}, \,\chi(g_1\circ g_2) \sigma_{g_2}(\chi(g_1)^{-1}) \chi(g_2)^{-1}\big), ~\textrm{since $\sigma_h=\epsilon_h$ for all $h\in H$}\\
			&=& \tau^\kappa \big((h_1,g_1),(h_2,g_2)\big) \, \big(\tau^\varphi((h_1,g_1),(h_2,g_2))\big)^{-1}, ~\textrm{by \eqref{tau 1 prime tau 2 prime inverse 1}}.
		\end{eqnarray*}
		
		Next, note that
		\begin{equation}\label{tilde f first term}
			\theta \big((h_1,g_1)\odot(h_2,g_2)\big)^{-1}=\theta(h_1\circ h_2,h_1\circ g_2\circ h_1^\dagger \circ g_1)^{-1}= \big(1, ~ \chi(h_1\circ g_2\circ h_1^\dagger \circ g_1)\big)
		\end{equation}
		and
		$$\widetilde{\xi}^{-1}_{(h_1,g_1)} \big(\theta(h_1,g_1)\big)= \widetilde{\xi}_{(h_1^\dagger,h_1^\dagger \circ g_1^\dagger\circ h_1)} \big(1,\, \chi(g_1)^{-1}\big)=
		\big(1,\,\epsilon_{g_1^\dagger\circ h_1}(\chi(g_1)^{-1}) \big),~~\text{using \eqref{tilde xi}.}
		$$
		This gives
		\begin{eqnarray}
			&& \widetilde{\xi}_{(h_1,g_1)\odot(h_2,g_2)}(\widetilde{\epsilon}_{(h_2,g_2)}(\widetilde{\xi}^{-1}_{(h_1,g_1)}(\theta(h_1,g_1)))) \nonumber\\
			&=&\widetilde{\xi}_{(h_1\circ h_2,h_1\circ g_2\circ h_1^\dagger \circ g_1)}(\widetilde{\epsilon}_{(h_2,g_2)}(1,\,\epsilon_{g_1^\dagger\circ h_1}(\chi(g_1)^{-1})))\nonumber\\
			&=&\big(1, ~ \epsilon_{h_2^\dagger\circ g_2\circ h_1^\dagger\circ g_1}(\epsilon_{h_2}(\epsilon^{-1}_{g_2}( \epsilon_{g_1^\dagger\circ h_1}(\chi(g_1)^{-1}))))\big),~\text{using \eqref{tilde xi} and \eqref{tilde epsilon}.}\nonumber\\
			&=& \big(1, ~\chi(g_1)^{-1}\big),~\textrm{since $\epsilon : H^{(\circ)} \rightarrow \Aut(I)$ is an anti-homomorphism}.\label{tilde f second term}
		\end{eqnarray}
		Further, we have
		\begin{equation}\label{tilde f third term}
			\widetilde{\xi}_{(h_1,g_1)} \big(\theta(h_2,g_2)\big) =\widetilde{\xi}_{(h_1,g_1)} \big(1,\, \chi(g_2)^{-1} \big)=\big(1, ~ \epsilon_{h_1^\dagger\circ g_1}(\chi(g_2)^{-1})\big),\text{ using \eqref{tilde xi}.}
		\end{equation}
		Using \eqref{tilde f first term}, \eqref{tilde f second term} and \eqref{tilde f third term}, we obtain
		\begin{eqnarray*}
			&& \theta \big((h_1,g_1)\odot(h_2,g_2)\big)^{-1} ~ \widetilde{\xi}_{(h_1,g_1)\odot(h_2,g_2)} \big(\widetilde{\epsilon}_{(h_2,g_2)}(\widetilde{\xi}^{-1}_{(h_1,g_1)}(\theta(h_1,g_1)))\big) ~ \widetilde{\xi}_{(h_1,g_1)}\big(\theta(h_2,g_2)\big)\\
			&=& \big(1, ~ \chi(h_1\circ g_2\circ h_1^\dagger \circ g_1)\big) \big(1, ~\chi(g_1)^{-1}\big) \big(1, ~ \epsilon_{h_1^\dagger\circ g_1}(\chi(g_2)^{-1})\big)\\
			&=& \big(1, ~ \chi(h_1\circ g_2\circ h_1^\dagger \circ g_1) \chi(g_1)^{-1} \epsilon_{h_1^\dagger\circ g_1}(\chi(g_2)^{-1})\big)\\
			&=& \big(1, ~ \chi(h_1\circ g_2\circ h_1^\dagger \circ g_1) \chi(g_1)^{-1} \sigma_{h_1^\dagger\circ g_1}(\chi(g_2)^{-1})\big), ~\textrm{since $\sigma_h=\epsilon_h$ for all $h\in H$}\\
			&=& \beta^{r^\kappa} \big((h_1,g_1),(h_2,g_2) \big)\,  \big(\tau'^\varphi((h_1,g_1),(h_2,g_2)) \big)^{-1}, ~\textrm{by \eqref{tau 1 prime tau 2 prime inverse 2}}.
		\end{eqnarray*}
		This proves our claim, and hence the proof of the theorem is complete.
	\end{proof}

	\section{Isoclinism of squares of skew left braces}\label{section isoclinism of squares}
	Recall from~\cite[Definition 6, 7]{MR3917122} that the \textit{annihilator} $\Ann(H)$ of a~skew left brace $(H, \cdot, \circ)$ is defined as
	\[
	\Ann(H)= \ker (\lambda^H) \cap \Z(H^{(\cdot)}) \cap \Fix (\lambda^H)=\{ a \in H \mid b \circ a = a \circ b = b \cdot a = a \cdot b~ \textrm{for all}~ b \in H \}.
	\]
	Similarly, the \textit{commutator} $H'$ of a~skew left brace $(H, \cdot, \circ)$ is defined as the subgroup of $H^{(\cdot)}$ generated by the commutator subgroup of $H^{(\cdot)}$ and the set (see~\cite[Definition~2.1]{MR4698318})
	\[
	\{a^{-1} \cdot ( a \circ b ) \cdot b^{-1} \mid a, b \in H \}.
	\]
	It turns out that both the annihilator $\Ann(H)$ and the commutator $H'$ are ideals of $(H, \cdot, \circ)$ (see~\cite[p.4]{MR3917122} and~\cite[p.2892]{MR4698318}, respectively). We now recall the definition of isoclinism of skew left braces from~\cite[Definition 2.7]{MR4698318}.
	
	\begin{definition}
		Two skew left braces $A$ and $B$ are said to be isoclinic if there exist skew brace isomorphisms $\xi_1: A/ \Ann(A) \to B/ \Ann(B)$ and $\xi_2: A' \to B'$ such that the diagram
		\begin{align}
			\begin{CD}
				A' @<{\theta_A}<< \big(A/ \Ann(A)\big) \times \big(A/ \Ann(A)\big)@>{\theta^{*}_A}>> A' \\
				@V{\xi_2}VV @V{\xi_1 \times \xi_1}VV @V{\xi_2}VV\\
				B' @<{\theta_B}<< \big(B/ \Ann(B)\big) \times \big(B/ \Ann(B)\big)@>{\theta^{*}_B}>> B'
			\end{CD}
		\end{align}
		commutes. Here, the maps $\theta_A, \theta^*_A: (A/ \Ann(A)) \times (A/ \Ann(A)) \rightarrow A' $ are defined by
		\[
		\theta_A(\overline{a}, \overline{b})=a\cdot b\cdot a^{-1}\cdot b^{-1} \text{ and }  \theta^*_A(\overline{a}, \overline{b})=\lambda^H_a(b)\cdot b^{-1}
		\]
		for all $a,b \in A$. The maps $\theta_B$ and $\theta^*_B$ are defined similarly.
	\end{definition}
	
	The following results from~\cite[Proposition 3.4 and Theorem 3.7]{RSU24} will be used in the sequel.
	
	\begin{proposition}\label{center and commutator of tilde}
		Let $(H, \cdot, \circ)$ be a~skew left brace and $(\widetilde{H},\bullet,\odot)$ be its square. Then the following assertions hold:
		\begin{enumerate}
			\item $\Z(\widetilde{H}^{(\bullet)}) \trianglelefteq \Fix(\lambda^H ) \times \Z(H^{(\circ)}) $.
			\item $(\widetilde{H}^{(\bullet)})'=H' \rtimes_{\lambda^H} (H^{(\circ)})'$.
		\end{enumerate}
		
	\end{proposition}
	
	The following observations are direct.
	
	\begin{lemma}
		Let $(H, \cdot, \circ)$ be a~skew left brace and $(\widetilde{H},\bullet,\odot)$ be its square. Then the following assertions hold:
		\begin{enumerate}
			\item $\Ann(H)\times \Ann(H)\subseteq \Ann(\widetilde{H})$.
			\item $\Ann(H)\times \Ann(H)$ is an ideal of $(\widetilde{H},\bullet,\odot)$.
		\end{enumerate}
		
	\end{lemma}
	
	\begin{lemma}\label{annihilator lemma}
		Let $(H,\cdot,\circ)$ be a~skew left brace and $(\widetilde{H},\bullet,\odot)$ be its square. Then the following assertions hold:
		\begin{enumerate}
			\item $\Ann(\widetilde{H}) \leq \Fix(\lambda^H ) \times \Z(H^{(\circ)})$.
			\item $\widetilde{H}/ \big(\Ann(H) \times \Ann(H) \big) \cong \big(\widetilde{H/ \Ann(H)}\big)$.
		\end{enumerate}
		
	\end{lemma}
	
	\begin{proof}
		By definition, we have $\Ann(\widetilde{H})=\Ker(\lambda^{\widetilde{H}})\cap \Z(\widetilde{H}^{(\bullet)})\cap \Fix(\lambda^{\widetilde{H}})$. By Proposition~\ref{center and commutator of tilde}(1), we have $\Z(\widetilde{H}^{(\bullet)}) \trianglelefteq \Fix(\lambda^H ) \times \Z(H^{(\circ)})$, and assertion (1) follows.
		
		Let $\overline{\lambda}$ be the induced action for $H/\Ann(H)$. It is easy to see that the map
		\[
		\phi: \widetilde{H}/ \big(\Ann(H) \times \Ann(H)\big) \to \big(\widetilde{H/ \Ann(H)}\big),
		\]
		given by $\phi(\overline{(a,b)}) = (\overline{a}, \overline{b})$ for all $a, b \in H$, is an isomorphism of skew left braces~\cite[Remark~3.6]{RSU24}.
	\end{proof}
	
	\begin{proposition}
		Let $(A,\cdot,\circ)$ and $(B,\cdot,\circ)$ be isoclinic skew left braces via an isoclinism $(\xi_1, \xi_2 )$. Then the following assertions hold:
		\begin{enumerate}
			\item There exists a~skew brace isomorphism
			\[
			\widetilde{\xi_1}: \widetilde{A}/\big(\Ann(A)\times\Ann(A)\big)\to \widetilde{B}/\big(\Ann(B)\times\Ann(B)\big),
			\]
			defined as $\widetilde{\xi_1} \big(\overline{a}_1, \overline{a}_2 \big)= \big(\xi_1 (\overline{a}_1 ),\xi_1 (\overline{a}_2 ) \big)$.
			\item There exists a~skew brace isomorphism
			\[
			\widetilde{\xi_2}:(\widetilde{A})' \to (\widetilde{B})',
			\]
			defined as $\widetilde{\xi_2}(a_1, a_2)=(\xi_2 (a_1), \xi_2 (a_2))$.
		\end{enumerate}
		
	\end{proposition}
	
	\begin{proof}
		In view of Lemma~\ref{annihilator lemma}(2), we identify
		$\widetilde{A}/ (\Ann(A) \times \Ann(A))$ with $\big(\widetilde{A/ \Ann(A)}\big) $, and similarly for $B$. It now follows that the map $\widetilde{\xi_1}(\overline{a}_1,\overline{a}_2 )=(\xi_1 (\overline{a}_1 ),\xi_1 (\overline{a}_2 ))$ is an isomorphism.
		
		By Proposition~\ref{center and commutator of tilde}(2), we have $(\widetilde{A}^{(\bullet)})'=A' \rtimes_{\lambda^A} (A^{(\circ)})'$ and $(\widetilde{B}^{(\bullet)})'=B' \rtimes_{\lambda^B} (B^{(\circ)})'$. By~\cite[Theorem 3.13]{RSU24}, the isomorphism $\xi_2:A' \to B'$ induces an isomorphism of groups $\widetilde{\xi_2}:(\widetilde{A}^{(\bullet)})' \to (\widetilde{B}^{(\bullet)})'$, given by $\widetilde{\xi_2}(a_1,a_2 )=(\xi_2 (a_1 ),\xi_2 (a_2 ))$. A direct check shows that $\widetilde{\xi_2}$ restricts to an isomorphism $(\widetilde{A})' \to (\widetilde{B})'$.
	\end{proof}
	
	\begin{theorem}\label{theorem isoclinicsdp}
		Let $(A,\cdot,\circ)$ and $(B,\cdot,\circ)$ be isoclinic skew left braces. If
		\[
		\Ann(\widetilde{A}) = \Ann(A) \times \Ann(A)
		\quad \text{and} \quad
		\Ann(\widetilde{B}) = \Ann(B) \times \Ann(B),
		\]
		then $(\widetilde{A},\bullet,\odot)$ and $(\widetilde{B},\bullet,\odot)$ are also isoclinic.
	\end{theorem}
	
	\begin{proof}
		Let $(\xi_1, \xi_2 )$ be an isoclinism from $(A,\cdot,\circ)$ onto $(B,\cdot,\circ)$. We claim that the following diagram commutes
		\begin{align}
			\begin{CD}
				(\widetilde{A})' @<{\theta_{\widetilde{A}}}<<\big(\widetilde{A}/ \Ann(\widetilde{A})\big) \times \big(\widetilde{A}/ \Ann(\widetilde{A})\big)@>{\theta^{*}_{\widetilde{A}}}>> (\widetilde{A})' \\
				@V{\widetilde{\xi_2}}VV @V{\widetilde{\xi_1} \times \widetilde{\xi_1}}VV @V{\widetilde{\xi_2}}VV\\
				(\widetilde{B})' @<{\theta_{\widetilde{B}}}<< \big(\widetilde{B}/ \Ann(\widetilde{B})\big) \times \big(\widetilde{B}/ \Ann(\widetilde{B})\big)@>{\theta^{*}_{\widetilde{B}}}>> (\widetilde{B})',
			\end{CD}
		\end{align}
		where the maps $\theta_{\widetilde{A}}, \theta^*_{\widetilde{A}}: \big(\widetilde{A}/\Ann(\widetilde{A})\big) \times \big(\widetilde{A}/ \Ann(\widetilde{A})\big) \rightarrow (\widetilde{A})' $ are defined by
		\[
		\theta_{\widetilde{A}} \big((\overline{a}_1,\overline{b}_1), (\overline{a}_2,\overline{b}_2) \big)=(a_1,b_1) \bullet (a_2,b_2) \bullet (a_1,b_1)^{-1}\bullet (a_2,b_2)^{-1}
		\]
		and
		\[
		\theta^*_{\widetilde{A}} \big((\overline{a}_1,\overline{b}_1), (\overline{a}_2,\overline{b}_2) \big)=(1,b_1^\dagger\circ a_1)\bullet (a_2,b_2)\bullet (1,b_1^\dagger\circ a_1)^{-1} \bullet (a_2,b_2)^{-1}
		\]
		for all $a_1,b_1,a_2,b_2 \in A$. The maps $\theta_{\widetilde{B}}$ and $\theta^*_{\widetilde{B}}$ are defined similarly.
		
		First, we prove the commutativity of the left hand side of the diagram. To avoid complexity of the notation, we will denote the composition $A\to A/\Ann(A)\to B/\Ann(B)$, where $A\to A/\Ann(A)$ is the canonical map, by $\xi_1$. We see that
		\begin{eqnarray}
			\theta_{\widetilde{A}} \big((\overline{a}_1,\overline{b}_1 ), (\overline{a}_2,\overline{b}_2 ) \big)&=&(a_1,b_1 ) \bullet (a_2,b_2 ) \bullet (a_1,b_1 )^{-1}\bullet (a_2,b_2 )^{-1}\nonumber\\
			&=& \big(a_1\lambda^A_{b_1}(a_2 ), \,b_1\circ b_2 \big)\bullet \big(\lambda^A_{b_1^\dagger}(a_1 )^{-1}, \,b_1^\dagger \big)\bullet \big(\lambda^A_{b_2^\dagger}(a_2 )^{-1}, \,b_2^\dagger \big)\nonumber\\
			&=& \big(a_1\lambda^A_{b_1}(a_2 ), \,b_1\circ b_2 \big)\bullet \big(\lambda^A_{b_1^\dagger}(a_1 )^{-1}\lambda^A_{b_1}(\lambda^A_{b_2^\dagger}(a_2 )^{-1}), \,b_1^\dagger\circ b_2^\dagger \big)\nonumber\\
			&=&\big(a_1\lambda^A_{b_1}(a_2 )\lambda^A_{b_1\circ b_2\circ b_1^\dagger\circ b_2^\dagger}(\lambda_{b_2}(a_1 )^{-1} a_2^{-1}), \,b_1\circ b_2\circ b_1^\dagger\circ b_2^\dagger\big)\label{theta left}.
		\end{eqnarray}
		Thus, we have
		\begin{eqnarray}
			&& \widetilde{\xi_2} \, \theta_{\widetilde{A}} \big((\overline{a}_1,\overline{b}_1 ), \, (\overline{a}_2,\overline{b}_2 ) \big)\nonumber\\
			&=&\widetilde{\xi_2}\big(a_1\lambda^A_{b_1}(a_2 )\lambda^A_{b_1\circ b_2\circ b_1^\dagger\circ b_2^\dagger}(\lambda_{b_2}(a_1 )^{-1} a_2^{-1}), b_1\circ b_2\circ b_1^\dagger\circ b_2^\dagger\big),\quad \text{using~\eqref{theta left}}\nonumber\\
			&=&\big(\xi_2 (a_1\lambda^A_{b_1}(a_2 )\lambda^A_{b_1\circ b_2\circ b_1^\dagger\circ b_2^\dagger}(\lambda_{b_2}(a_1 )^{-1} a_2^{-1})),~\xi_2 (b_1\circ b_2\circ b_1^\dagger\circ b_2^\dagger)\big)\label{theta left 1}.
		\end{eqnarray}
		The isoclinism $(\xi_1, \xi_2)$ from $(A, \cdot, \circ)$ onto $(B, \cdot, \circ)$ gives
		\begin{eqnarray}
			\xi_2 (a\cdot b\cdot a^{-1}\cdot b^{-1}) &= &\xi_1 (a)\cdot \xi_1 (b)\cdot \xi_1 (a)^{-1}\cdot\xi_1 (b)^{-1} \label{sbisoclinism xi1}~\textrm{and}\\
			\lambda^B_{\xi_1 (a)} \big(\xi_1 (b)\big) \cdot \big(\xi_1 (b) \big)^{-1} &=& \xi_2 \big(\lambda^A_a (b) \cdot b^{-1} \big)\label{sbisoclinism xi2}
		\end{eqnarray}
		for all $a, b \in A$. We compute the first coordinate of~\eqref{theta left 1} and obtain
		\begin{align*}
			\xi_2 & (a_1\lambda^A_{b_1}(a_2 )\lambda^A_{b_1\circ b_2\circ b_1^\dagger\circ b_2^\dagger}(\lambda_{b_2}(a_1 )^{-1} a_2^{-1}))\\
			&=\xi_2 (a_1\lambda^A_{b_1}(a_2 )(\lambda^A_{b_1\circ b_2\circ b_1^\dagger\circ b_2^\dagger}(a_2\lambda^A_{b_2}(a_1 ) ))^{-1})\\
			&=\xi_2 (a_1\lambda^A_{b_1}(a_2 )(a_2\lambda^A_{b_2}(a_1 ))^{-1} a_2\lambda^A_{b_2}(a_1 ) (\lambda^A_{b_1\circ b_2\circ b_1^\dagger\circ b_2^\dagger}(a_2\lambda^A_{b_2}(a_1 ) ))^{-1})\\
			&=\xi_2\big(a_1\lambda^A_{b_1}(a_2 )(a_2\lambda^A_{b_2}(a_1 ))^{-1} \big)~\xi_2\big(a_2\lambda^A_{b_2}(a_1 )(\lambda^A_{b_1\circ b_2\circ b_1^\dagger\circ b_2^\dagger}(a_2\lambda^A_{b_2}(a_1 ) ))^{-1}\big)\\
			&=\xi_2\big(a_1\lambda^A_{b_1}(a_2 ) a_1^{-1} \lambda^A_{b_1}(a_2 )^{-1} \lambda^A_{b_1}(a_2 ) a_1\lambda^A_{b_2^\dagger}(a_2\lambda^A_{b_2}(a_1 ))^{-1} \lambda^A_{b_2^\dagger}(a_2\lambda^A_{b_2}(a_1 )) (a_2\lambda^A_{b_2}(a_1 ))^{-1} \big)\\
			&\qquad\xi_1 (a_2 )\lambda^B_{\xi_1 (b_2 )}(\xi_1 (a_1 ))(\lambda^B_{\xi_1 (b_1\circ b_2\circ b_1^\dagger\circ b_2^\dagger)}(\xi_1 (a_2 )\lambda^B_{\xi_1 (b_2 )}(\xi_1 (a_1 )) ))^{-1},\text{ using \eqref{sbisoclinism xi2}}\\
			&=\xi_2\big(a_1\lambda^A_{b_1}(a_2 ) a_1^{-1} \lambda^A_{b_1}(a_2 )^{-1}\big) \, \xi_2\big(\lambda^A_{b_1}(a_2 ) a_1\lambda^A_{b_2^\dagger}(a_2\lambda^A_{b_2}(a_1 ))^{-1}\big) \\
			&\qquad \xi_2\big(\lambda^A_{b_2^\dagger}(a_2\lambda^A_{b_2}(a_1 )) (a_2\lambda^A_{b_2}(a_1 ))^{-1} \big) \, \xi_1 (a_2 ) \, \lambda^B_{\xi_1 (b_2 )}(\xi_1 (a_1 )) \\
			&\qquad (\lambda^B_{\xi_1 (b_1 )\circ\xi_1 ( b_2 )\circ \xi_1 (b_1 )^\dagger\circ \xi_1 (b_2 )^\dagger}(\xi_1 (a_2 )\lambda^B_{\xi_1 (b_2 )}(\xi_1 (a_1 )) ))^{-1}, \text{ by \cite[Proposition 3.9]{MR4698318}}\\
			&=\xi_1 (a_1 )\lambda^B_{\xi_1 (b_1 )}(\xi_1 (a_2 )) \xi_1 (a_1 )^{-1} \lambda^B_{\xi_1 (b_1 )}(\xi_1 (a_2 ))^{-1} \\
			&\qquad\xi_2\big(\lambda^A_{b_1}(a_2 ) a_2^{-1} a_2 \lambda^A_{b_2^\dagger}(a_2 )^{-1}\big) \lambda^B_{\xi_1 (b_2 )^\dagger}(\xi_1 (a_2 )\lambda^B_{\xi_1 (b_2 )}(\xi_1 (a_1 ))) (\xi_1 (a_2 )\lambda^B_{\xi_1 (b_2 )}(\xi_1 (a_1 )))^{-1} \\
			&\qquad\xi_1 (a_2 )\lambda^B_{\xi_1 (b_2 )}(\xi_1 (a_1 ))(\lambda^B_{\xi_1 (b_1 )\circ\xi_1 ( b_2 )\circ \xi_1 (b_1 )^\dagger\circ \xi_1 (b_2 )^\dagger}(\xi_1 (a_2 )\lambda^B_{\xi_1 (b_2 )}(\xi_1 (a_1 )) ))^{-1}, \\
			&\qquad \text{using \eqref{sbisoclinism xi1} and \eqref{sbisoclinism xi2}} \\
			&=\xi_1 (a_1 )\lambda^B_{\xi_1 (b_1 )}(\xi_1 (a_2 )) \xi_1 (a_1 )^{-1} \lambda^B_{\xi_1 (b_1 )}(\xi_1 (a_2 ))^{-1} \\
			&\qquad\lambda^B_{\xi_1 (b_1 )}(\xi_1 (a_2 )) \, \xi_1 (a_2 )^{-1} \, \xi_1 (a_2 ) \, \lambda^B_{\xi_1 (b_2 )^\dagger}(\xi_1 (a_2 ))^{-1} \, \lambda^B_{\xi_1 (b_2 )^\dagger}(\xi_1 (a_2 )\lambda^B_{\xi_1 (b_2 )}(\xi_1 (a_1 ))) \\
			&\qquad (\xi_1 (a_2) \lambda^B_{\xi_1 (b_2 )}(\xi_1 (a_1 )))^{-1} \, \xi_1 (a_2) \, \lambda^B_{\xi_1 (b_2 )}(\xi_1 (a_1 )) \\
			&\qquad (\lambda^B_{\xi_1 (b_1 )\circ\xi_1 ( b_2 )\circ \xi_1 (b_1 )^\dagger\circ \xi_1 (b_2 )^\dagger}(\xi_1 (a_2 )\lambda^B_{\xi_1 (b_2 )}(\xi_1 (a_1 )) ))^{-1}, \text{using \eqref{sbisoclinism xi2}} \\
			&=\xi_1 (a_1 )\lambda^B_{\xi_1 (b_1 )}(\xi_1 (a_2 ))\lambda^B_{\xi_1 (b_1 )\circ\xi_1 ( b_2 )\circ \xi_1 (b_1 )^\dagger\circ \xi_1 (b_2 )^\dagger}(\lambda^B_{\xi_1 (b_2 )}(\xi_1 (a_1 ))^{-1}\xi_1 (a_2 )^{-1} ).
		\end{align*}
		Using this in~\eqref{theta left 1}, we get
		\begin{eqnarray*}
			&& \widetilde{\xi_2} \, \theta_{\widetilde{A}} \big((\overline{a}_1,\overline{b}_1 ), \, (\overline{a}_2,\overline{b}_2 ) \big)\nonumber\\
			&=&\big(\xi_2 (a_1\lambda^A_{b_1}(a_2 )\lambda^A_{b_1\circ b_2\circ b_1^\dagger\circ b_2^\dagger}(\lambda_{b_2}(a_1 )^{-1} a_2^{-1})),\xi_2 (b_1\circ b_2\circ b_1^\dagger\circ b_2^\dagger)\big)\\
			&=&\big(\xi_1 (a_1 )\lambda^B_{\xi_1 (b_1 )}(\xi_1 (a_2 ))\lambda^B_{\xi_1 (b_1 )\circ\xi_1 ( b_2 )\circ \xi_1 (b_1 )^\dagger\circ \xi_1 (b_2 )^\dagger}(\lambda^B_{\xi_1 (b_2 )}(\xi_1 (a_1 ))^{-1}\xi_1 (a_2 )^{-1} ),\\
			&&\xi_1 (b_1 )\circ \xi_1 (b_2 )\circ \xi_1 (b_1 )^\dagger\circ \xi _1 (b_2 )^\dagger \big),\quad \text{by~\cite[Proposition 3.9]{MR4698318}}\\
			&=&(\xi_1 (a_1 ),\xi_1 (b_1 )) \bullet (\xi_1 (a_2 ),\xi_1 (b_2 )) \bullet (\xi_1 (a_1 ),\xi_1 (b_1 )) ^{-1}\bullet (\xi_1 (a_2 ),\xi_1 (b_2 ))^{-1}\nonumber\\
			&=&\theta_{\widetilde{B}}(\widetilde{\xi_1}\times\widetilde{\xi_1})\big((\overline{a}_1,\overline{b}_1 ),(\overline{a}_1,\overline{b}_1 )\big), \quad \text{using Lemma~\eqref{annihilator lemma}(2)}.
		\end{eqnarray*}
		This shows that the left side of the diagram commutes. For the right hand diagram, observe that
		\begin{align}
			\theta^*_{\widetilde{A}} \big((\overline{a}_1,\overline{b}_1 ), (\overline{a}_2,\overline{b}_2 ) \big)&= \big(1,b_1^\dagger\circ a_1\big) \bullet \big(a_2,b_2\big) \bullet\big(1,b_1^\dagger\circ a_1 \big)^{-1} \bullet \big(a_2,b_2 \big)^{-1}\nonumber\\
			&= \big(\lambda^A_{b_1^\dagger\circ a_1}(a_2 ), b_1^\dagger\circ a_1\circ b_2 \big) \bullet \big(\lambda^A_{a_1^\dagger\circ b_1}(\lambda^A_{b_2^\dagger}(a_2^{-1})), a_1^\dagger\circ b_1\circ b_2^\dagger \big)\nonumber\\
			&= \big(\lambda^A_{b_1^\dagger\circ a_1}(a_2 ) \cdot \lambda^A_{b_1^\dagger\circ a_1\circ b_2}(\lambda^A_{a_1^\dagger\circ b_1\circ b_2^\dagger}(a_2^{-1})), b_1^\dagger\circ a_1\circ b_2\circ a_1^\dagger\circ b_1 \circ b_2^\dagger \big)\nonumber\\
			&= \big(\lambda^A_{b_1^\dagger\circ a_1}(a_2 ) \cdot \lambda^A_{b_1^\dagger\circ a_1\circ b_2\circ a_1^\dagger\circ b_1 \circ b_2^\dagger}(a_2^{-1}), b_1^\dagger\circ a_1\circ b_2\circ a_1^\dagger\circ b_1 \circ b_2^\dagger \big).\label{isoisb 1}
		\end{align}
		Thus, we have
		\begin{align*}
			\widetilde{\xi_2}&\theta^*_{\widetilde{A}} \big((\overline{a}_1,\overline{b}_1 ), \, (\overline{a}_2,\overline{b}_2 ) \big)\\
			&=	\widetilde{\xi_2} \big(\lambda^A_{b_1^\dagger\circ a_1}(a_2 )\cdot \lambda^A_{b_1^\dagger\circ a_1\circ b_2\circ a_1^\dagger\circ b_1 \circ b_2^\dagger}(a_2^{-1}), b_1^\dagger\circ a_1\circ b_2\circ a_1^\dagger\circ b_1 \circ b_2^\dagger \big), \text{ using \eqref{isoisb 1}}\\
			&= \big(\xi_2 \big(\lambda^A_{b_1^\dagger\circ a_1}(a_2 ) \cdot \lambda^A_{b_1^\dagger\circ a_1\circ b_2\circ a_1^\dagger\circ b_1 \circ b_2^\dagger}(a_2^{-1})\big),~\xi_2\big(b_1^\dagger\circ a_1\circ b_2\circ a_1^\dagger\circ b_1 \circ b_2^\dagger \big) \big)\\
			&= \big(\xi_2 \big(\lambda^A_{b_1^\dagger\circ a_1}(a_2 )\cdot a_2^{-1}\cdot a_2 \cdot \lambda^A_{b_1^\dagger\circ a_1\circ b_2\circ a_1^\dagger\circ b_1 \circ b_2^\dagger}(a_2^{-1}) \big), \\
			&\qquad \xi_1 (b_1^\dagger\circ a_1 )\circ\xi_1 ( b_2 )\circ \xi_1 (a_1^\dagger\circ b_1 ) \circ \xi_1 (b_2 )^\dagger \big), \text{ by \cite[Proposition 3.9]{MR4698318}}\\
			&= \big(\xi_2 \big(\lambda^A_{b_1^\dagger\circ a_1}(a_2 )\cdot a_2^{-1} \cdot (\lambda^A_{b_1^\dagger\circ a_1\circ b_2\circ a_1^\dagger\circ b_1 \circ b_2^\dagger}(a_2 )\cdot a_2^{-1})^{-1} \big), \\
			&\qquad\xi_1 (b_1^\dagger\circ a_1 )\circ\xi_1 ( b_2 )\circ \xi_1 (a_1^\dagger\circ b_1 ) \circ \xi_1 (b_2 )^\dagger \big)\\
			&= \big(\lambda^B_{\xi_1 (b_1^\dagger\circ a_1 )}(\xi_1 (a_2 )) \cdot \xi_1 (a_2 )^{-1} \cdot (\lambda^B_{\xi_1 (b_1^\dagger\circ a_1 )\circ \xi_1 (b_2 )\circ \xi_1 (a_1^\dagger\circ b_1 ) \circ \xi_1 (b_2 )^\dagger}(\xi_1 (a_2 )) \cdot \xi_1 (a_2 )^{-1})^{-1},\\
			&\qquad\xi_1 (b_1 )^\dagger\circ \xi_1 (a_1 )\circ\xi_1 ( b_2 )\circ \xi_1 (a_1 )^\dagger\circ\xi_1 ( b_1 ) \circ \xi_1 (b_2 )^\dagger \big), \text{ using \eqref{sbisoclinism xi2}}\\
			&= \big(\lambda^B_{\xi_1 (b_1 )^\dagger\circ \xi_1 (a_1 )}(\xi_1 (a_2 )) \cdot \lambda^B_{\xi_1 (b_1 )^\dagger\circ\xi_1 (a_1 )\circ \xi_1 (b_2 )\circ (\xi_1 (b_1 )^\dagger\circ\xi_1 (a_1 ))^\dagger \circ \xi_1 (b_2 )^\dagger}(\xi_1 (a_2 )^{-1}),~\\
			&\qquad\xi_1 (b_1 )^\dagger\circ \xi_1 (a_1 )\circ\xi_1 ( b_2 )\circ (\xi_1 (b_1 )^\dagger\circ\xi_1 (a_1 ))^\dagger \circ \xi_1 (b_2 )^\dagger \big)\\
			&= \big(1,~\xi_1 (b_1 )^\dagger\circ \xi_1 (a_1 ) \big) \bullet \big(\xi_1 (a_2 ), \, \xi_1 (b_2 )\big) \bullet \big(1,\, \xi_1 (b_1 )^\dagger\circ \xi_1 (a_1 )\big)^{-1} \bullet \big(\xi_1 (a_2 ),\, \xi_1 (b_2 )\big)^{-1}\\
			&=\theta^*_{\widetilde{B}} \big( \big(\xi_1 (\overline{a}_1 ), \xi_1 (\overline{b}_1 ) \big), \big(\xi_1 (\overline{a}_2 ), \xi_1 (\overline{b}_2 ) \big) \big), \text{ using Lemma \eqref{annihilator lemma}(2)}\\
			&= \theta^*_{\widetilde{B}}\, (\widetilde{\xi_1} \times \widetilde{\xi_1} ) \big((\overline{a}_1,\overline{b}_1 ), (\overline{a}_2,\overline{b}_2 ) \big).
		\end{align*}
		This completes the proof of the theorem.
	\end{proof}

		
\end{document}